\documentclass[reqno]{amsart}

\usepackage{a4wide,color,eucal,enumerate,mathrsfs,amssymb}
\usepackage{color}
\usepackage{mathrsfs,enumerate,ifthen,boxedminipage,fancybox}

\title{Lecture notes on gradient flows and optimal transport}

\author{Sara Daneri}
\address{ S.I.S.S.A., Via Beirut 2-4, {34151}, Trieste, Italy}
\email{daneri@sissa.it}
\author{Giuseppe Savar\'e}
\address{Dipartimento di Matematica, Universit\`a di Pavia. Via Ferrata, 1 -- 27100 Pavia, Italy.}
\email{giuseppe.savare@unipv.it}
\urladdr{http://www.imati.cnr.it/\textasciitilde savare}

\newtheorem{theorem}{Theorem}[section]

\newtheorem{corollary}[theorem]{Corollary}
\newtheorem{lemma}[theorem]{Lemma}
\newtheorem{proposition}[theorem]{Proposition}
\newtheorem{definition}[theorem]{Definition}

\theoremstyle{definition}

\newtheorem{example}[theorem]{Example}

\theoremstyle{remark}
\newtheorem{remark}[theorem]{Remark}

\newcommand{\N}{\mathbb{N}}

\newcommand{\R}{\mathbb{R}}

\newcommand{\BB}{\mathscr{B}}

\newcommand{\LL}{\mathscr{L}}

\newcommand{\PP}{\mathscr{P}}

\newcommand{\VV}{\mathscr{V}}
\newcommand{\WW}{\mathscr{W}}

\newcommand{\cP}{{\ensuremath{\mathcal P}}}

\newcommand{\bb}{{\mbox{\boldmath$b$}}}

\newcommand{\ii}{{\mbox{\boldmath$i$}}}

\newcommand{\rr}{{\mbox{\boldmath$r$}}}

\renewcommand{\tt}{{\mbox{\boldmath$t$}}}

\newcommand{\vv}{{\mbox{\boldmath$v$}}}

\newcommand{\ggamma}{{\mbox{\boldmath$\gamma$}}}

\newcommand{\eeta}{{\mbox{\boldmath$\eta$}}}

\newcommand{\mmu}{{\mbox{\boldmath$\mu$}}}

\newcommand{\ttau}{{\mbox{\boldmath$\tau$}}}

\newcommand{\sfd}{{\sf d}}

\newcommand{\sfE}{{\sf E}}

\newcommand{\sfS}{{\sf S}}

\newcommand{\rmd}{{\mathrm d}}
\newcommand{\rme}{{\mathrm e}}

\newcommand{\rmD}{{\mathrm D}}

\newcommand{\Kliminf}{K\kern-3pt-\kern-2pt\mathop{\rm lim\,inf}\limits}  
\newcommand{\supp}{\mathop{\rm supp}\nolimits}   
\newcommand{\argmin}{\mathop{\rm argmin}\limits}   
\renewcommand{\d}{{\mathrm d}}
\newcommand{\dt}{{\d t}}

\newcommand{\restr}[1]{\lower3pt\hbox{$|_{#1}$}}
\newcommand{\topref}[2]{\stackrel{\eqref{#1}}#2}
\newcommand{\Leb}[1]{{\mathscr L}^{#1}}      
\newcommand{\la}{{\langle}}                  
\newcommand{\ra}{{\rangle}}
\newcommand{\down}{\downarrow}              
\newcommand{\up}{\uparrow}
\newcommand{\eps}{\varepsilon}  
\newcommand{\nchi}{{\raise.3ex\hbox{$\chi$}}}

\newcommand{\Rd}{{\R^d}}

\newcommand{\forevery}{\text{for every }}

\renewcommand{\d}{{\rmd}}
\renewcommand{\dt}{{\rmd t}}

\newcommand{\FlowName}{\ensuremath{\mathsf S}}
\newcommand{\Flow}[2]{\FlowName_{#1}(#2)}
\newcommand{\EVIshort}[1]{\ensuremath{\EVIname_{#1}}}
\newcommand{\EVI}[4]{\ensuremath{\EVIname_{#4}(#1,#2,#3)}}
\newcommand{\EVIname}{\ensuremath{\mathrm{EVI}}}
\newcommand{\Dom}[1]{\ensuremath{D(#1)}}
\newcommand{\DistSquare}[2]{\DistName^2(#1,#2)}               
\newcommand{\Dist}[2]{\DistName(#1,#2)}               
\newcommand{\DistName}{{\sf d}}          
\newcommand{\Urd}{{\frac{ \d}{\d t}}}
\renewcommand{\sfE}[2]{{\sf E}_{#1}(#2)}
\newcommand{\DomainSlope}[1]{D(|\partial#1|)}                 
\newcommand{\MetricSlope}[2]{|\partial#1|(#2)}              
\newcommand{\MetricSlopeSquare}[2]{|\partial#1|^2(#2)}      
\newcommand{\umin}{{\bar u}}
\newcommand{\Pc}[2]{\overline{#1}\kern-2pt^{\vphantom 0}_{#2}}
\newcommand{\taueta}{{\tau,\eta}}

\newcommand{\phinfty}{\phi^\infty}
\newcommand{\xinfty}{x^\infty}

\newcommand{\barxinfty}{\bar x^\infty}

\newcommand{\freccia}{\rightarrow}
\newcommand{\barra}{\backslash}

\renewcommand{\u}{u\sp{\prime}}
\renewcommand{\v}{v\sp{\prime}}

\numberwithin{equation}{section}

\newcommand{\UU}{\mathscr{U}}
\newcommand{\XX}{\mathscr{X}}
\newcommand{\mtE}{\mathsf{E}}
\newcommand{\mtS}{\mathsf{S}}
\newcommand{\rL}{\mathit{L}}

\newcommand{\Id}{\mbox{\boldmath${i}$}}
\renewcommand{\ttau}{\tau}
\newcommand{\unt}{U^{n}_{\ttau}}

\newcommand{\Ed}[2]{|#1-#2|}
\newcommand{\EdSquare}[2]{|#1-#2|^2}

\begin{document}
\begin{abstract}
  We present a short overview on the strongest variational formulation for
  gradient flows of geodesically $\lambda$-convex functionals in
  metric spaces, with applications to diffusion equations in
  Wasserstein spaces of probability measures.
  These notes are based on a series of lectures given by the second
  author for the Summer School ``Optimal transportation: Theory and applications'' in Grenoble during the week of June 22-26, 2009.
\end{abstract}
\maketitle

\section*{Introduction}
These notes are based on a series of lectures given by the second
  author for the Summer School ``Optimal transportation: Theory and
  applications'' in Grenoble during the week of June 22-26, 2009.

  We try to summarize some of the main results concerning gradient
  flows of geodesically $\lambda$-convex functionals in metric spaces
  and applications to diffusion PDE's in the Wasserstein space of
  probability measures.
  Due to obvious space constraints, the theory and the references
  presented here are largely incomplete and should be intended as a{n}
  over-simplified presentation of a quickly evolving subject.
  We refer to the books \cite{Ambrosio-Gigli-Savare05,Villani09}
  for a detailed account of 
  the large literature available on these topics.

  In the first section we collect some elementary and well known results
  concerning gradient flows of smooth convex functions in $\Rd$.
  We selected just a few topics, which are well suited for a
  ``metric'' formulation and provide a useful guide for the more
  abstract developments.

  In the second section we present the main (and strongest) notion of
  gradient flow in metric spaces {characterized} by the solution of a
  metric \emph{evolution
    variational inequality}: the aim here is to show the consequence
  of this definition, without any assumptions on the space and on the
  functional (except completeness and lower semicontinuity):
  we shall see that solutions to evolution variational inequalities enjoy nice
  stability, asymptotic, and regularization properties. We also investigate
  the relationships with two different approaches, 
  \emph{curves of maximal slope} and \emph{minimizing movements},
  and we discuss a first stability result with respect to
  perturbations of the generating functional with respect to
  $\Gamma$-convergence.

  The third section is devoted to some fundamental generation results
  for gradient flows of geodesically $\lambda$-convex functional{s}:
  here we adopt the method of \emph{minimizing
    movement} to construct suitable families of discrete approximating
  solutions
  and we show three basic convergence results.
  
  Apart from Sections \ref{subsec:stability} (stability of gradient
  flows with respect to $\Gamma$-convergence of the functionals)
  and \ref{subsec:compact} (existence of
  curves of maximal slope), we made a substantial effort to avoid
  any compactness argument in the theory, which is mainly focused on
  purely metric arguments. So we will present a slightly relaxed
  version of the minimizing movement scheme, which is always solvable
  by invoking Ekeland's variational principle, and
  the main existence and generation results for $\lambda$-gradient
  flows 
  rely on refined Cauchy estimates and
  crucial geometric assumptions on the distance of the metric
  space.

  The last section is devoted to applications of the metric theory to
  evolution equations in the so called ``Wasserstein spaces''
  $\PP_2(X)$ of probability measures. We recall a few basic facts about such
  spaces, the characterization of geodesics and absolutely continuous
  curves, and some geometric {properties} of the Wasserstein distance.
  Three basic examples of (or, better, displacement-)
  $\lambda$-convex functionals in $\PP_2(\Rd)$ are presented, together
  with the evolutionary PDE's they are associated with.
  A short account of possible extensions of the theory to
  measure-metric spaces concludes the notes.

\begin{section}{Gradient flows for smooth $\lambda$-convex functions in the Euclidean space}
In this section we recall some simple properties of the gradient
flow of a $C^2$
function $\varphi:\R^d\freccia\R$ satisfying the global lower bound
$\rmD^2\varphi\geq\lambda I$ for some $\lambda\in\R$. We will focus on
those aspects which {rely} just on the ``metric'' structure of $\R^d$
and therefore could make sense in more general metric spaces.
We denote by $\sfd(u,v)=|u-v|$ the Euclidean distance on $\R^d$
induced by the scalar product
$\la\cdot,\cdot\ra$.
\begin{remark}[A few basic facts about $\lambda$-convex functions]
  \upshape
  We wil{l} extensively use the following well known equivalent {characterizations} of a $\lambda$-convex function $\varphi:\R^d\to\R$
  (here $x,x_0,x_1$ are
  arbitrary points in $\R^d$)
  \begin{subequations}
    \begin{description}
    \item[Hessian inequality]
      \begin{equation}
        \label{eq:1}
        \rmD^2\varphi(x)\ge \lambda I\quad
        \text{i.e.}\quad\langle\rmD^2\varphi(x)\xi,\xi\rangle\ge
        \lambda|\xi|^2
        \quad\text{for every }\xi\in
        \R^d.
      \end{equation}
     \item[$\lambda$-monotonicity of $\nabla\varphi$]
        \begin{equation}
          \label{eq:2}
          \langle \nabla\varphi(x_0)-\nabla\varphi(x_1),x_0-x_1\rangle\ge
          \lambda|x_0-x_1|^2.
        \end{equation}
        \item[$\lambda$-convexity inequality]
          \begin{equation}
            \label{eq:3}
            \varphi(x_\theta)\le
            (1-\theta)\varphi(x_0)+\theta\varphi(x_1){-\frac{\lambda}{2}}\theta(1-\theta)|x_0-x_1|^2\quad
            \text{$x_\theta:=(1-\theta)x_0+\theta x_1$, $\theta\in [0,1].$}
          \end{equation}
          \item[Sub-gradient inequality]
            \begin{equation}
              \label{eq:4}
              \la\nabla\varphi(x_1),x_1-x_0\ra-\frac\lambda2|x_1-x_0|^2\ge\varphi(x_1)-\varphi(x_0)\ge\langle\nabla\varphi(x_0),x_1-x_0\rangle+\frac\lambda
              2|x_1-x_0|^2.
            \end{equation}
    \end{description}
    Notice that
    \begin{equation}
      \label{eq:18}
    \text{$\varphi$ is $\lambda$-convex if and only if
        $\tilde\varphi(x):=\varphi(x)-\frac\lambda2|x|^2$ is convex.}
    \end{equation}
    In particular, there exist constants $a\in \R,$ $\bb\in \R^d$  such that
    \begin{equation}
      \label{eq:14}
      \varphi(x)\ge a+\la \bb,x\ra+\frac\lambda2|x|^2.
    \end{equation}
  \end{subequations}
\end{remark}
\begin{definition}[Gradient flow]
  The gradient flow of $\varphi$ is the family of 
  maps
  \begin{equation*}
    \mtS_t:\R^d\freccia\R^d,\quad t\in[0,+\infty),
  \end{equation*}
  characterized by the following property: for every $u_0\in \R^d$, 
  $\mtS_0(u_0):=u_0$ and the
  curve $u_t:=\mtS_t(u_0)$, $t\in (0,+\infty)$, is the unique $C^1$ 
  solution of the Cauchy problem
  \begin{equation}\label{gf}
    \frac\d\dt u_t=-\nabla\varphi(u_t)\quad
    \text{in }(0,+\infty),\quad
    \lim_{t\downarrow0}u_t=u_0.
  \end{equation}
\end{definition}

\noindent By the standard Cauchy-Lipschitz theory and the
\emph{a priori} estimates we will show in the next {t}heorem, for every
initial datum $u_0\in \R^d$ equation
\eqref{gf} admits a unique global solution so that the family
$\mtS_t$, $t\in [0,+\infty)$,
is a continuous semigroup of Lipschitz maps,  thus
satisfying
\begin{equation}
  \label{eq:5}
  \mtS_{t+h}(u_0)=\mtS_t(\mtS_h(u_0)),\quad
  \lim_{t\downarrow0}\mtS_t(u_0)=\mtS_0(u_0)=u_0\quad\forevery u_0\in \R^d.
\end{equation}
\subsection{Basic estimates}
\begin{theorem}[Basic differential estimates]
  Let us assume that $\varphi\in C^2(\Rd)$ is $\lambda$-convex;
  if $u:[0,+\infty)\to\R^d$ is a solution of \eqref{gf} then
  \begin{align}
    \label{eq:6}
    \frac\d\dt \frac 12\EdSquare{u_t}v+
    \frac
    \lambda 2\EdSquare{u_t}v=
    \rme^{-\lambda t}\frac\d\dt\Big(\rme^{\lambda t}\frac
    12\EdSquare{u_t}v\Big)&\le \varphi(v)-\varphi(u_t)\quad\forevery
    v\in \R^d,
    \tag{EVI$_\lambda$}
    \\\label{eq:7}
    \frac\d\dt \varphi(u_t)=-|u_t'|^2=-|\nabla\varphi(u_t)|^2&\le
    0,\tag{EI}
    \\\label{eq:8}
    \frac\d\dt \Big(\rme^{{2}\lambda t}\,|\nabla\varphi(u_t)|{^2}\Big)=
    \frac\d\dt \Big(\rme^{{2}\lambda t}\,|u_t'|{^2}\Big)&\le 0;\tag{SI$_\lambda$}
    \intertext{moreover, if $v$ is another solution to \eqref{gf}
      then}
    \frac\d\dt \Big(\rme^{\lambda t}\,\Ed{u_t}{v_t}\Big)&\le 0.
    \tag{Cont$_\lambda$}
    \label{eq:9}
  \end{align}
\end{theorem}
\begin{proof}
  We sketch here the easy calculations.\\  
  For the \emph{evolution variational inequality} \eqref{eq:6}:
  \begin{displaymath}
    \frac\d\dt\frac12\EdSquare{u_t}v=\la u_t',u_t-v\ra\topref{gf}=\la
    \nabla\varphi(u_t),v-u_t\ra
    \topref{eq:4}\le \varphi(v)-\varphi(u_t)-\frac\lambda 2\EdSquare{u_t}v.
  \end{displaymath}
  The \emph{energy identity} \eqref{eq:7}:
  \begin{displaymath}
    \frac\d\dt\varphi(u_t)=\la\nabla\varphi(u_t),u_t'\ra\topref{gf}=-|\nabla\varphi(u_t)|^2\topref{gf}=-|u_t'|^2.    
  \end{displaymath}
  The \emph{slope inequality} \eqref{eq:8}:
  \begin{displaymath}
    \frac\d\dt |\nabla\varphi(u_t)|^2=
    2\la \rmD^2\varphi(u_t)\nabla\varphi(u_t),u_t'\ra\topref{gf}=
    -2\la \rmD^2\varphi(u_t)\nabla\varphi(u_t),\nabla\varphi(u_t)\ra
    \topref{eq:1}\le -2\lambda |\nabla\varphi(u_t)|^2.
  \end{displaymath}
  The \emph{$\lambda$-contraction} property \eqref{eq:9}:
  \begin{displaymath}
    \frac\d\dt \EdSquare{u_t}{v_t}=2\la
    u_t'-v_t',u_t-v_t\ra\topref{gf}=
    -2\la \nabla\varphi(u_t)-\nabla\varphi(v_t),u_t-v_t\ra
    \topref{eq:2}\le -2\lambda \EdSquare{u_t}{v_t}.\qedhere
  \end{displaymath}
\end{proof}

In order to write in a simple way suitable integrated versions of the
previous inequalities, we set
\begin{equation}
  \label{eq:cap2:11}
  \sfE\lambda t:=\int_0^t \rme^{\lambda r}\,\d r=
  \begin{cases}
    \frac{\rme^{\lambda t}-1}\lambda&\text{if }\lambda\neq0,\\
    t&\text{if }\lambda=0.
  \end{cases}
\end{equation}
\begin{corollary}[Pointwise and integral inequalities]
  If $u:[0,+\infty)\to \R^d$ is a solution to \eqref{gf} then
  \begin{align}
    \label{eq:10}
    \frac{\rme^{\lambda t}}2\EdSquare{u_t}v+\sfE\lambda t\big(\varphi(u_t)-\varphi(v)\big)+\frac{\big(\sfE\lambda t\big)^2}2|\nabla\varphi(u_t)|^2
    &\le \frac 12 \EdSquare{u_0}v
    \quad\forevery v\in \R^d,\\
    \label{eq:11}
    \varphi(u_t)+\frac 12\int_0^t
    \Big(|u_r'|^2+|\nabla\varphi(u_r)|^2\Big)\,\d r&=\varphi(u_0),\\
    \label{eq:12}
    |\nabla\varphi(u_t)|&\le \rme^{-\lambda t}|\nabla\varphi(u_0)|;\\
   \intertext{moreover, if $v$ is another solution to \eqref{gf} then}
   \label{eq:13}
   \Ed{u_t}{v_t}&\le \rme^{-\lambda t}\Ed{u_0}{v_0}.
  \end{align}
  In particular, when $\lambda>0$, $\varphi$ admits a unique minimum
  point $\umin $ and
  \begin{gather}
    \label{eq:15}
    \frac\lambda 2\EdSquare{u_t}{\umin }\le
    \varphi(u_t)-\varphi(\umin )\le \frac
    1{2\lambda}|\nabla\varphi(u_t)|^2\\
    \label{eq:16}
    \Ed{u_t}{\umin }\le \rme^{-\lambda t}\Ed{u_0}{\umin},\quad
    \varphi(u_t)-\varphi(\umin )\le \rme^{-2\lambda
      t}{\big(\varphi(u_0)-\varphi(\umin )\bigr)}.
  \end{gather}
\end{corollary}
\begin{proof}
  We have just to check \eqref{eq:10}: if $A_t$ denotes the quantity
  in the left-hand side, we show that $A_t$ is nonincreasing. A
  differentiation in time yields
  \begin{align*}
    \frac\d\dt A_t&=
    \rme^{\lambda t}\Big(\frac \lambda2
    \EdSquare{u_t}v+\frac\d\dt\frac12\EdSquare{u_t}v+\varphi(u_t)-\varphi(v)+\sfE\lambda t|\nabla\varphi(u_t)|^2\Big)\\&\qquad
    +\sfE\lambda t\frac\d\dt\varphi(u_t)+\frac{\big(\sfE\lambda t\big)^2}2\frac\d\dt|\nabla\varphi(u_t)|^2\Big)\\&
    \topref{eq:6}\le
    \sfE\lambda t\Big(\rme^{\lambda
      t}|\nabla\varphi(u_t)|^2+\frac\d\dt\varphi(u_t)+\frac{\sfE\lambda t}2\frac\d\dt|\nabla\varphi(u_t)|^2\Big)
    \\&\topref{eq:7}=
    \sfE\lambda t\Big(
    \big(\rme^{\lambda
      t}-1\big)|\nabla\varphi(u_t)|^2+\frac{\sfE\lambda t}2\frac\d\dt|\nabla\varphi(u_t)|^2\Big)
    \\&\topref{eq:8}\le
    \sfE\lambda t\Big(
    \big(\rme^{\lambda
      t}-1-\lambda \sfE\lambda t\big)|\nabla\varphi(u_t)|^2\Big)
    \topref{eq:cap2:11}=0.\qedhere
  \end{align*}
\end{proof}
In terms of the maps $\mtS_t${,} \eqref{eq:13} yields the
$\lambda$-contraction estimate
\begin{equation}\label{contr}
\sfd(\mtS_t(u_0),\mtS_t(v_0))\leq \rme^{-\lambda
  t}\sfd(u_0,v_0)\quad\forevery u_0,v_0\in\R^d,\ t\geq0,
\end{equation}
thus showing the Lipschitz property of $\mtS_t$ and 
the uniqueness and continuous dependence w.r.t.\ the
initial data of the solutions of \eqref{gf}.

\subsection{Approximation by the Implicit Euler scheme}
One of the simplest but very useful way{s} to construct discrete
approximations of the solution to
\eqref{gf} (and to show its existence by a limiting process) is given
by the \emph{implicit Euler scheme.}


For a given time step $\ttau>0$ we consider the associated uniform
partition {of} $[0,+\infty)$
\begin{equation}
  \label{eq:19}
\mathcal{P}_{\ttau}:=\{0=t^0_{\ttau}<t^1_{\ttau}<...<t^n_{\ttau}<...\}, \quad t^n_{\ttau}:=n\tau,
\end{equation}
and we look for a discrete sequence $(U^n_\tau)_{n\in \N}$
whose value $U^n_\tau$ should provide an effective approximation of
$u(t^n_\tau)$. $U^n_\tau$ are defined recursively, starting from a
suitable choice of $U^0_\tau\approx u_0$, by solving at each step the
equation of the Euler scheme
\begin{align}\label{eus}
\frac{U^n_{\ttau}-U^{n-1}_{\ttau}}{\ttau}=-\nabla{\varphi}(U^n_{\ttau})\quad n=1,2,\cdots,
\end{align}
or, equivalently,
\begin{equation}
  \label{eq:20}
  U^n_\tau=J_\tau(U^{n-1}_\tau),\quad J_\tau:= (I+\tau\nabla\varphi)^{-1}.
\end{equation}
Existence of a discrete approximating solution can be easily obtained
by looking for the minimizers of the function
\begin{equation}
  \label{eq:21}
  U\mapsto \Phi(\tau,U^{n-1}_\tau;U):=\frac
  1{2\tau}\big|U-U^{n-1}_\tau\big|^2+\varphi(U).
\end{equation}
In fact, it is immediate to check that any minimizer $U^n_\tau$ of \eqref{eq:21}
solves \eqref{eus}; moreover, the function defined by \eqref{eq:21} is
$(\tau^{-1}+\lambda)$-convex and therefore it admits a unique
minimizer whenever $\tau^{-1}>-\lambda$.

Denoting by $U_\tau:[0,+\infty)\to\R^d$ the piecewise linear
interpolant of the discrete values $(U^n_\tau)_{n\in \N}$ on the grid
$\cP_\tau$, defined by
\begin{equation}
  \label{eq:22}
  U_\tau(t):=\frac{t-t^{n-1}_\tau}\tau
  U^{n-1}_\tau+\frac{t^n_\tau-t}\tau U^n_\tau\quad\text{if }t\in [t^{n-1}_\tau,t^n_\tau],
\end{equation}
one expects that $U_\tau(t)$ converges to the solution $u_t$ to
\eqref{gf} as $\tau\down0$.
\begin{theorem}
  If $\lim_{\tau\downarrow0}U^0_\tau=u_0$ then the family of
  piecewise linear interpolants $(U_\tau)_{\tau>0}$ satisfies the Cauchy
  condition as $\tau\down0$ with respect to the uniform convergence on each compact
  interval $[0,T]$, $T>0$; its unique limit is the solution
  $u_t$ of \eqref{gf}.
  Moreover, for every $T>0$ there exists a universal constant $C(\lambda,T)$ such that
  \begin{equation}
    \label{eq:23}
    \sup_{t\in [0,T]} \Ed{u_t}{U_\tau(t)}\le \Ed{u_0}{U^0_\tau}+ C(\lambda,T) |\nabla\varphi(u_0)|\,\tau.
  \end{equation}
  In particular, when $\lambda=0$ we can choose $C=\frac 1{\sqrt 2}$,
  independent of $T$.
\end{theorem}
\begin{proof}[Remarks about the proof]
  In the present finite dimensional smooth setting, the proof of {the}
  convergence of $U_\tau$ is not difficult: considering e.g.\ the case
  $\lambda=0$, we can apply the contraction property of the map $J_\tau$
  defined by \eqref{eq:20}
  \begin{equation}
    \label{eq:24}
    \Ed{J_\tau(x)}{J_\tau(y)}\le \Ed xy\quad\forevery x,y\in \R^d,
  \end{equation}
  to obtain the uniform bound
  \begin{equation}
    \label{eq:25}
    \tau^{-1}\Ed{U^n_\tau}{U^{n-1}_\tau}=|\nabla\varphi(U^n_\tau)|\le
    |\nabla\varphi(U^{n-1}_\tau)|\quad\forevery n\ge 1,
  \end{equation}
  so that
  \begin{equation}
    \label{eq:26}
    |U'_\tau(t)|\le \sup_{n\in \N}\tau^{-1}
    \Ed{U^n_\tau}{U^{n-1}_\tau}=
    \tau^{-1}\Ed{U^1_\tau}{U^0_\tau}\le
    |\nabla\varphi(U^0_\tau)|\quad\text{for every }t\in
    [0,+\infty)\setminus \cP_\tau.
  \end{equation}
  Since
  $\lim_{\tau\down0}|\nabla\varphi(U^0_\tau)|=|\nabla\varphi(u_0)|$ it
  follows that
  $(U_\tau)_{\tau>0}$ satisfies a uniform Lipschitz condition and therefore it admits a
  suitable subsequence uniformly converging to a Lipschitz curve $u$
  in each compact interval $[0,T]$. Denoting by $\bar U_\tau(t)$ the
  piecewise constant interpolant
  \begin{equation}
    \label{eq:27}
    \bar U_\tau(t):=U^n_\tau\quad\text{if }t\in (t^{n-1}_\tau,t^n],
  \end{equation}
  the same estimate \eqref{eq:26} shows that
  \begin{equation}
    \label{eq:28}
    \sup_{t\in (0,+\infty)}\Ed{U_\tau(t)}{\bar U_\tau(t)}\le \tau |\nabla\varphi(U^0_\tau)|,
  \end{equation}
  so that $\bar U_\tau$ has the same limit points than $U_\tau$. On
  the other hand, \eqref{eus} yields
  \begin{equation}
    \label{eq:29}
    U'_\tau(t)=-\nabla\varphi(\bar U_\tau(t))\quad\text{in
    }[0,+\infty)\setminus \cP_\tau,
  \end{equation}
  and we can pass to the limit in an integrated form of \eqref{eq:29}
  thus showing that $u$ solves \eqref{gf}.

  The uniform error estimate \eqref{eq:23} is subtler: a simple
  derivation in the case $\lambda=0$ can be found in
  \cite{Nochetto-Savare-Verdi00}, see also \cite{Rulla96,Savare96}. Its main functional interest relies
  on the fact that it involves just the lower bound on the Hessian of
  $\varphi$ but not its upper bound (and therefore, it does not require a uniform
  Lipschitz condition on $\nabla\varphi$).
\end{proof}

\subsection{Metric characterization of Gradient flows in $\R^d$}
The energy  identity \eqref{eq:7} (with his integrated version
\eqref{eq:11}) and the evolution variational inequality \eqref{eq:6}
not only provide important estimates on the solution to \eqref{gf} but
can also be used to characterize it.

Concerning \eqref{eq:7} we can even relax the identity, as the
following proposition shows.
\begin{proposition}[Curves of maximal slope]
  \label{gfprop} A $C^1$ curve
  $u:[0,+\infty)\freccia\R^d$
  is a solution to \eqref{gf} if and only if it satisfies the
 \emph{Energy Dissipation Inequality}
\begin{equation}\label{ei}
  \frac{\d}{\dt}\varphi(u_t)\le
  -\frac{1}{2}|u_t'|^2-\frac{1}{2}|\nabla\varphi(u_t)|^2\quad\text{in }(0,+\infty)
  \tag{EDI}
\end{equation}
or its weaker integrated form
\begin{equation}
  \label{eq:17}
  \varphi(u_t)+\frac12\int_0^t\Big(|u_r'|^2+|\nabla\varphi(u_r)|^2\Big)\,\d
  r\le \varphi(u_0)\quad\forevery t\in (0,+\infty).
    \tag{EDI'}
\end{equation}
\end{proposition}
\begin{proof}
Il $u$ is a $C^1$ curve the chain rule yields
\begin{align}\label{eigf}
  \varphi(u_t)&=\varphi(u_0)+\int_0^t\langle\nabla\varphi(u_r),u_r'\rangle\,\d r,
\end{align}
so that \eqref{eq:17} yields
\begin{displaymath}
  \frac12\int_0^t\big|u_r'+\nabla\varphi(u_r)\big|^2\,\d r=
  \frac12\int_0^t\Big(|u_r'|^2+|\nabla\varphi(u_r)|^2\Big)\,\d r+ \int_0^t\langle\nabla\varphi(u_r),u_r'\rangle\,\d r\le0,
\end{displaymath}
and therefore $\u_r=-\nabla\varphi(u_r)$ for $\Leb 1$-a.e.\ $r\in
(0,t)$. Since $t$ is arbitrary and $u\in C^1$, $u$ solves \eqref{gf}.
\end{proof}
Notice that in the previous formulation we did not use the
$\lambda$-convexity of $\varphi$: the argument only relies on the chain rule.

In the following proposition we show that also the evolution variational
inequality \eqref{eq:6} characterizes a solution of \eqref{gf}. In
fact, if \eqref{eq:6} admits a solution \emph{for every initial datum}
$u_0$, then $\varphi$ is $\lambda$-convex.

\begin{proposition}[Characterization of Gradient Flows through the EVI]\label{gfevi}
  If $u:[0,+\infty)\to \R^d$ is a $C^1$ curve solving \eqref{eq:6}
  then $u$ is a solution to \eqref{gf}.
\end{proposition}
\begin{proof}
  Applying the chain rule for the squared distance function
  $\frac12\EdSquare\cdot v$ we easily have
\begin{align}\label{subdiff}
  \la u_t',u_t-v\ra\le
  \varphi(v)-\varphi(u_t)-\frac\lambda2\EdSquare{u_t}v
  \quad\forevery v\in\R^d,\ t>0.
\end{align}
Choosing $v:=u_t+\eps \xi$, for $\eps>0$ and $\xi\in \R^d$ and
dividing by $\eps$ we obtain
\begin{displaymath}
  -\la u_t',\xi\ra\le
  \eps^{-1}\Big(\varphi(u_t+\eps\xi)-\varphi(u_t)\Big)-\frac{\lambda\eps}2|\xi|^2\quad
  \forevery \xi\in \R^d.
\end{displaymath}
Passing to the limit as $\eps\down0$ we eventually get
\begin{displaymath}
  -\la u_t',\xi\ra\le \la
  \nabla\varphi(u_t),\xi\ra\quad\forevery\xi\in \R^d,
\end{displaymath}
so that $-u_t'=\nabla\varphi(u_t)$.
\end{proof}

\begin{proposition}\label{eviconv1} Let us suppose that there exists a
  $C^1$ semigroup $\tilde\mtS_t:\R^d\to\R^d$, $t\ge0$, of smooth maps such that for every
  $u_0\in \R^d$ the curve $u_t:=\tilde\mtS_t(u_0)$ satisfies \eqref{eq:6}. Then $\varphi$ is $\lambda$-convex.
\end{proposition} 
\begin{proof}
  We consider for simplicity the case $\lambda=0$;
  for arbitrary $u^0,u^1\in \R^d$ we set
  \begin{displaymath}
    u^s:=(1-s)u^0+su^1,\quad u^s_t:=\tilde\mtS_t(u^s)
  \end{displaymath}
  and we want to show that
  \begin{displaymath}
    \frac\d{\d s}\varphi(u^s)\restr{s=0}\le \frac\d{\d s}\varphi(u^s)\restr{s=1}.
  \end{displaymath}
  We get
  \begin{align*}
    \frac\d{\d s}&\varphi(u^s)\restr{s=0}=
    \la \nabla\varphi(u^0),u^1-u^0\ra\topref{gf}=
    -\la \frac\d\dt u^0_t\restr{t=0},u^1-u^0\ra=
    \frac\d{\d
      t}\Big(\frac12\EdSquare{u^0_t}{u^1}\Big)\restr{t=0}
    \\&\topref{eq:6}\le\varphi(u^1)-\varphi(u^0)
    \topref{eq:6}\le -\frac\d{\d
      t}\Big(\frac12\EdSquare{u^0}{u^1_t}\Big)\restr{t=0}
    \topref{gf}\le\la \nabla\varphi(u^1),u^1-u^0\ra=
    \frac\d{\d s}\varphi(u^s)\restr{s=1}\qedhere
  \end{align*}
\end{proof}



\subsection{Extensions to more general functional settings}
The simple finite dimensional theory for smooth functionals has been
extended in various directions; without claiming any completeness, we
quote here four different points of view:
\paragraph{\emph{The theory of differential inclusions and maximal monotone
  operators in Hilbert spaces,}}
developed in the seventies by \textsc{Komura} \cite{Komura67},
\textsc{Crandall-Pazy} \cite{Crandall-Pazy69}, \textsc{Crandall-Liggett} \cite{Crandall-Liggett71},
\textsc{Br\'ezis}
\cite{Brezis71}, \textsc{B\'enilan} \cite{Benilan72}, \textsc{J.L.\ Lions} \cite{Lions69}:
we refer to the monographs
\cite{Brezis73,Barbu76,Lions69}.
In this framework one considers the gradient flow generated by a proper
lower semicontinuous $\lambda$-convex functional
$\phi:H\to(-\infty,+\infty]$, where $H$ is a separable Hilbert
space. By using tools of convex analysis, clever regularization
techniques, and replacing $\nabla\varphi$
with the multivalued subdifferential operator $\partial\phi$, one can
basically reproduce all the estimates and results we briefly
discussed in the finite dimensional setting which just depend on the
lower bound of the Hessian of $\varphi$, avoiding any strong compactness assumptions.

In this framework, the resolvent operator
$J_\tau:=(I+\tau\partial\phi)^{-1}$ is single-valued and non-expansive, i.e.
\begin{align}\label{contrh}
\sfd(J_{\ttau}[u], J_{\ttau}[v])\leq \sfd(u,v)\quad\forevery u,v\in H,\quad\ttau>0.
\end{align}
This property is the key ingredient to prove, as in the \textsc{Crandall-Ligget{t}} generation theorem
\cite{Crandall-Liggett71}, uniform convergence of the exponential formula
\begin{equation}\label{exp}
 u_t=\underset{n\rightarrow+\infty}{\lim}(J_{t/ n})^n[u_0],\quad \sfd(u_t,(J_{t/n})^n[u_0] )\leq\frac{2|\partial\phi|(u_0)t}{\sqrt{n}}
\end{equation}
 and therefore to define a contraction semigroup on $\overline{D(\phi)}$.

Being generated by a convex functional, this semigroup exhibits a nice regularization effect, since $u_t\in D(\partial\phi)$ even if $u_0\in\overline{D(\phi)}$. 
Moreover, the curve $u_t$ can be characterized as the unique solution of the evolution variational inequality
\eqref{eq:6}, whose formulation goes back to \cite{Lions-Stampacchia67}.
Optimal error estimates for the implicit Euler discretization in the spirit of \eqref{eq:23} have been obtained by
\cite{Baiocchi89,Rulla96,Savare96,Nochetto-Savare-Verdi00}.
\paragraph{\emph{The theory of the curves of maximal slope in metric
    spaces,}} developed in the eighties by \textsc{De Giorgi, Degiovanni, Marino, Tosques}
in a series of papers originating from
\cite{DeGiorgi-Marino-Tosques80,DeGiorgi-Degiovanni-Marino-Tosques83},
and culminating in
\cite{Degiovanni-Marino-Tosques85,Marino-Saccon-Tosques89} (but see also the
more recent
\cite{Cardinali-Colombo-Papalini-Tosques97} and the presentation
of \cite{Ambrosio95,Ambrosio-Gigli-Savare05}).
Here $\phi:X\to(-\infty,+\infty]$ is a proper and lower semicontinuous
functional defined in the complete metric space $X$ and 
one looks for absolutely continuous curves satisfying a
suitable form of the Energy dissipation inequality \eqref{ei}, where
$|u'|$ should be interpreted as the \emph{metric velocity of the curve
  $u$} and $|\nabla\varphi(u)|$ should be replaced by the \emph{metric
  slope of $\phi$}. The theory is usually based on local compactness
of the sublevels of $\phi$ and various kind of assumptions on its
slope, yielding in particular its lower semicontinuity and the
possibility to write a weak form of the chain rule.
The advantage of this approach relies on its flexibility, but in
general metric spaces uniqueness and stability properties of curves of
maximal slope are not known.

\paragraph{\emph{Limits of discrete solutions, generalized minimizing
    movements.}} This is the weakest approach, which has been
clarified in \cite{DeGiorgi93} and independently applied to
different kind of problems (see e.g.\ 
\cite{Luckhaus90}, \cite{Gianazza-Savare96}, \cite{Jordan-Kinderlehrer-Otto98},
\cite{Mielke-Theil-Levitas02}). It just provides a
general
approximating scheme which is quite useful to construct some limit
curves by compactness arguments, but one can hardly deduce refined
properties of
these curves from general metric results and each example {deserves} a
careful
ad hoc investigation.

\paragraph{\emph{Evolution variational inequalities in metric
    spaces:}}
this is the strongest point of view, which is related to the metric evolution variational inequality
\eqref{eq:6} and goes back to \textsc{B\'enilan} \cite{Benilan72} notion of integral solutions
to evolution equations in Banach spaces. Its application to gradient flows in metric spaces
has been developed in \cite{Ambrosio-Gigli-Savare05} and it will be adopted in these notes.

\end{section}
\begin{section}{Gradient flows and evolution variational inequalities in metric spaces}\label{sect:metricgf}

The aim of this section is to study
the metric notion of gradient flows associated to the (metric
formulation of the) evolution
variational inequality \eqref{eq:6}.
 
Throughout the rest of these notes, $(X,\sfd)$ will be a
\emph{complete and separable metric space} and
$\phi:X\freccia(-\infty,+\infty]$ a
proper and l.s.c. functional on $X$ with non empty
domain $D(\phi)=\{v\in X:\phi(v)<+\infty\}$. 
We will look for curves $u:[0,+\infty)\freccia X$ which satisfy
properties that depend only on the metric structure of $X$
and that in the case of a smooth function $\phi=\varphi$ on $X=\R^d$ satisfy the ODE \eqref{gf}.

\subsection{A few metric concepts}
Let us first recall the notion of \emph{metric velocity} and
\emph{metric slope} (see e.g.\ \cite{Ambrosio-Gigli-Savare05}).

\begin{definition}[Absolutely continuous curves] We say that a curve $v:(a,b)\subset\R\freccia X$ belongs to $AC^p_{(\mathrm{loc})}(a,b;X)$ for some $p\in[1,+\infty]$ if there exists $m\in\rL^p_{(\mathrm{loc})}(a,b)$ such that
\begin{equation}\label{accurv}
 {\Dist{v_s}{v_t}}\leq\int_s^tm(r)\,dr\qquad \forevery a<s\leq t <b.
\end{equation}
If $p=1$ we say that $v$ is a \emph{(locally) absolutely continuous curve}.
\end{definition}

\begin{theorem}[Metric derivative]
  If $v:(a,b)\freccia X$ is an absolutely continuous curve then the limit
\begin{equation}\label{metricder}
|\v|(t)=\underset{s\freccia t}{\lim}\frac{{\Dist{v_s}{v_t}}}{|t-s|}
\end{equation}
exists for $\LL^1$-a.e. $t\in(a,b)$ and it is called \emph{metric derivative} of $v$ at the point $t$. Moreover, the function $t\mapsto|\v|(t)$ belongs to $\rL^1(a,b)$, it is an admissible integrand for the right hand side of \eqref{accurv}, and it is minimal in the following sense:
\begin{align*}
|\v|(t)\leq m(t)\quad\text{for $\LL^1$-a.e. $t\in(a,b)$, for each function $m$ satisfying \eqref{accurv}.}
\end{align*}
\end{theorem}

\begin{definition}[Metric Slope] The metric slope of $\phi$ at a point $v\in X$ is given by
\begin{equation}
  |\partial\phi|(v)=
  \begin{cases}
      +\infty&\text{if }v\not\in \Dom\phi,\\
      0&\text{if $v\in \Dom\phi$ is isolated,}\\
      \displaystyle\limsup\limits_{w\to v}\frac{\big(\phi(v)-\phi(w)\big)^+}
      {\Dist vw}&\text{otherwise.}
     \end{cases}
\end{equation}
\end{definition}

\subsection{Structural properties of solutions to Evolution
  Variational Inequalities}

The next (quite restrictive) definition is modeled
on the case of $\lambda$-convex functionals in Euclidean-like spaces
and has been introduced and discussed in \cite[Chap.~4]{Ambrosio-Gigli-Savare05}.
\begin{definition}[\EVIname\ and Gradient flow]
  \label{def:GFlow}
  A solution of the evolution variational inequality
  $\EVI X\DistName\phi\lambda$, $\lambda\in \R$, is a
  locally absolutely continuous curve
  $u:t\in (0,+\infty)\mapsto u_t\in \Dom\phi $ such that
  \begin{equation}
    \label{eq:EVI}
    \frac 12 \Urd\DistSquare {u_t}v+
    \frac \lambda 2\DistSquare{u_t}v
    \le \phi(v)-\phi(u_t)
   \quad
   \Leb 1\text{-a.e.\ in }(0,+\infty),\quad
   \forevery v\in \Dom\phi .
    \tag{\EVIshort\lambda}
  \end{equation}
  A $\lambda$-gradient flow of $\phi$ is a family of continuous maps
  $\FlowName_t:\overline{D(\phi)}\to D(\phi)$, $t>0$, such that for every $u\in \overline{D(\phi)}$
  \begin{subequations}
    \begin{equation}
      \lim_{t\downarrow0}\Flow tu=u=:\Flow 0u,\qquad
      \Flow{t+h}u=\Flow h{\Flow tu}\quad\text{for every $t,h\ge 0,$}
      \label{eq:18bis}
  \end{equation}
  \begin{equation}
    \label{eq:28bis}
    \text{the curve }t\mapsto \Flow tu\quad\text{is  a solution of \EVI X\DistName\phi\lambda}.
  \end{equation}
  \end{subequations}
\end{definition}
The next result shows that \eqref{eq:EVI} can be formulated avoiding
differentiation and without
assuming the absolute continuity of $u$ (see \cite{Savare10} for the proof).
\begin{theorem}[Derivative free characterization of solutions to \eqref{eq:EVI}]
  \label{thm:uniqueness}
  A curve $u:(0,+\infty)\to
  \overline{\Dom\phi}$ is a solution of $\EVI X\DistName\phi\lambda$
  according to Definition
  \ref{def:GFlow} if and only if 
  for every $s,t\in (0,+\infty)$ with $s<t$ and $v\in \Dom\phi$
  \begin{equation}
    \label{eq:cap1:78bis}
    \frac {\rme^{\lambda(t-s)}}2\DistSquare{u_t}v-
    \frac 12\DistSquare{u_s}v\le
    \sfE\lambda{t-s}\Big(\phi(v)-\phi(u_t)\Big).
    \tag{$\EVIname_\lambda'$}
  \end{equation}
\end{theorem}
Notice that \eqref{eq:cap1:78bis} yields the pointwise right-upper
differential inequality
 \begin{equation}
    \label{eq:EVIevery}
    \frac 12 {\frac\d\dt\!\!}^+\DistSquare {u_t}v+
    \frac \lambda 2\DistSquare{u_t}v
    \le \phi(v)-\phi(u_t)
   \quad
   \forevery v\in \Dom\phi ,
  \end{equation}
  at \emph{every} time $t>0$: here $ {\frac\d\dt\!\!}^+\zeta$ denotes
  the right-upper Dini derivative
  $\limsup_{h\down0}h^{-1}(\zeta(t+h)-\zeta(t))$.
  
  The next result collects many useful properties of solutions to
  $\EVI X\DistName\phi\lambda$
(see \cite{Savare10} and an analogous result of \cite{Ambrosio-Savare06}
in the Wasserstein framework): they reproduce in the metric framework
the estimates of the previous section and show that \eqref{eq:EVI}
contains all the information concerning the gradient flow of $\phi$.
\begin{theorem}[Properties of solutions to \eqref{eq:EVI}]
  \label{thm:main1}
  Let $u, u^1,u^2:[0,+\infty)\to X$ be solutions of
  $\EVI X\DistName\phi\lambda$.
  \begin{description}
  \item[$\lambda$-contraction and uniqueness]
    \begin{equation}
    \label{eq:cap1:81}
    \Dist{u^1_t}{u^2_t}\le \rme^{-\lambda (t-s)}
    \Dist{u^1_s}{u^2_s}\quad
    \forevery\, 0\le s<t<+\infty.
  \end{equation}
  In particular, for every $u_0\in \overline{\Dom\phi }$
  there is at most one solution $u$ of $\EVI X\DistName\phi\lambda$ satisfying
  the initial condition $\lim_{t\down0}u_t=u_0$.
\item[Regularizing effects]
  $u$ is \emph{locally Lipschitz continuous} in $(0,+\infty)$ and
   $u_t\in \DomainSlope\phi \subset \Dom\phi $
   \emph{for every $t>0$}.
   Moreover in the time interval $[0,+\infty)$
  \begin{gather}
     \text{the map $t \mapsto \phi(u_t)$ is
       non{-}increasing and (locally semi-, if $\lambda<0$) convex},
     \label{eq:Cap1:58}\\
     \label{eq:cap2:3}
     \text{the map $t\mapsto \rme^{\lambda t}
       \MetricSlope\phi{u_t}$ is non{-}increasing and right continuous,}
   \end{gather}
   the following 
   regularization/a priori estimate hold{s}
      \begin{equation}
        \label{eq:2bis}
        \frac {\rme^{\lambda t}}2\DistSquare{u_t}v+
        \sfE\lambda t\Big(\phi(u_t)-\phi(v)\Big)+\frac {\big(\sfE \lambda t\big)^2}2\MetricSlopeSquare\phi{u_t}\le
        \frac 12\DistSquare{u_0}v
      \end{equation}
      for every $v\in \Dom\phi $; in particular
    \begin{align}
      \label{eq:Cap12:111}
      \phi(u_t)&\le
      \phi(v)+
      \frac{1}{2 \sfE\lambda t}\DistSquare{u_0}{v},\\
      \label{eq:Cap1:38}
      \MetricSlopeSquare\phi{u_t}&\le
      \frac{1}{2\rme^{\lambda t}-1}\MetricSlopeSquare\phi v+
      \frac1{(\sfE\lambda t)^2}\DistSquare{u_0}v\quad
      \text{if }-\lambda t<\log 2.
    \end{align}
    \item[Asymptotic expansion for $t\downarrow0$]
      If $u_0\in \DomainSlope\phi $ and $\lambda\le 0$ then for every $v\in \Dom\phi $ and $t\ge0$
      \begin{equation}
        \label{eq:3bis}
        \frac {\rme^{2\lambda t}}2\DistSquare{u_t}v-
        \frac 12\DistSquare{u_0}v\le \sfE{2\lambda} t\big(\phi(v)-\phi(u_0)\big)+
        \frac {t^2}2\MetricSlopeSquare\phi{u_0}.
      \end{equation}
  \item[Right and left limits, energy identity]
  For every $t>0$ the \emph{right limits}
  \begin{equation}
    \label{eq:Cap12:23}
    |\dot u_{t+}|:=\lim_{h\down0}\frac{\Dist{u_t}{u_{t+h}}}h,\quad
    \frac{\d}{\d t}\phi(u_{t+}):=
    \lim_{h\down0}\frac{\phi(u_{t+h})-\phi(u_t)}h
  \end{equation}
  exist, they satisfy
  \begin{equation}
      \label{eq:Cap12:24}
      \frac{\d}{\d t}\phi(u_{t+})=-|\dot u_{t+}|^2=
      -\MetricSlopeSquare\phi{u_t},      
    \end{equation}
    and they define a right continuous map.
    \eqref{eq:Cap12:23} and \eqref{eq:Cap12:24} hold
    at $t=0$ iff $u_0\in \DomainSlope\phi $.
    Moreover, there exists an at most countable set $\mathcal C\subset
    (0,+\infty)$ such that
    the analogous identities for the \emph{left} {\emph{limits}}
    hold for every $t\in (0,+\infty)\setminus\mathcal C$.
  \item[Asymptotic behavior]
    If $\lambda>0$, then
    $\phi$ admits a unique minimum point $\umin $
    and for every  $t\geq t_0\ge0$ we have
    \begin{subequations}
      \begin{gather}
        \label{eq:Cap1:42}
        \frac\lambda 2\DistSquare{u_t}{\umin }\le
        \phi(u_t)-\phi(\umin )\le \frac1{2\lambda}
        |\partial\phi|^2(u_t),
        \\
        \label{eq:Cap12:18}
        \DistSquare{u_t}{\umin }\le
        \DistSquare{u_{t_0}}{\umin } \rme^{-\lambda
          (t-t_0)},\quad
        \\
        \label{eq:Cap1:40}
        \phi(u_t)-\phi(\umin )\le
        \Big(\phi(u_{t_0})-\phi(\umin )\Big)\rme^{-2\lambda(t-t_0)},
        \quad
        \phi(u_t)-\phi(\umin )\le
        \frac{1}{2}\rme^{\lambda{{(t-t_0)}}}\DistSquare{u_{t_0}}{\umin },
        \\
        \label{eq:Cap1:41}\MetricSlope\phi{u_t}\le
        \MetricSlope\phi{u_{t_0}}\rme^{-\lambda(t-t_0)},\quad
        \MetricSlope\phi{u_t}\le
        \frac1{\rme^{\lambda {{(t-t_0)}}}\Dist{u_{t_0}}{\umin }}.
      \end{gather}
    \end{subequations}
  If $\lambda=0$ and $\umin $ is any minimum point of $\phi$ then
  we have
  \begin{equation}
    \label{eq:Cap12:19}
    \begin{gathered}
      \MetricSlope\phi{u_t}\le \frac{\DistSquare
        {u_0}{\overline u}}t,
      \quad
      \phi(u_t)-\phi(\umin )\le \frac{
        \DistSquare{u_0}{\umin }}{2t},
      \\
      \text{the map}\quad t\mapsto \DistSquare{u_t}{\umin }
      \quad
      \text{is not increasing.}
  \end{gathered}
\end{equation}
  \item[Continuity of the energy and the slope]
    If $u^n\in C^0([0,+\infty); X)$ are solutions of
    \\$\EVI X\DistName\phi\lambda$
    such that $\lim_{n\up+\infty}u^n_0=u_0$,
    then
    \begin{align}
      \label{eq:cap2:16}
      \lim_{n\up+\infty}\phi(u^n_{t})&=\phi(u_t)&&\forevery t>0,\\
      \label{eq:80}
      \lim_{n\up+\infty}|\partial\phi|(u^n_{t})&=|\partial\phi|(u_t)
      &&\forevery
      t\in (0,+\infty)\setminus\mathcal C.
    \end{align}
  \end{description}
\end{theorem}
\begin{proof}[We just sketch the proof of the contraction property \eqref{eq:cap1:81}.]
  For a fixed $s\in(0,+\infty)$ we have that
  \begin{equation}\label{ts}
    \frac{\partial}{\partial
      t}\frac{1}{2}\sfd^2(u^1_t,u^2_s)+\frac{\lambda}{2}
    \sfd^2(u^1_t,u^2_s)\leq\phi(u^2_s)-\phi(u^1_t)\quad\forevery t\in(0,+\infty),
  \end{equation}
  while for a fixed $t\in(0,+\infty)$
  \begin{equation}\label{st}
    \frac{\partial}{\partial s}\frac{1}{2}\sfd^2(u^1_t,u^2_s)+\frac{\lambda}{2}\sfd^2(u^1_t,u^2_s)\leq\phi(u^1_t)-\phi(u^2_s)\quad\forevery s\in(0,+\infty).
  \end{equation}
  Adding \eqref{ts} and \eqref{st} we get
  \begin{equation*}
    \frac{\partial}{\partial t}\frac{1}{2}\sfd^2(u^1_t,u^2_s)+\frac{\partial}{\partial s}\frac{1}{2}\sfd^2(u^1_t,u^2_s)+\lambda \sfd^2(u^1_t,u^2_s)\leq0;
  \end{equation*}
  Applying \cite[Lemma 4.3.4]{Ambrosio-Gigli-Savare05} we obtain
  \begin{equation*}
    \frac{\d}{\dt}\sfd^2(u^1_t,u^2_t)\leq-2\lambda
    \sfd^2(u^1_t,u^2_t)\quad\text{$\Leb 1$-a.e.\ in }(0,+\infty)
  \end{equation*}
  and therefore we obtain \eqref{eq:cap1:81}.
\end{proof}

Theorem \ref{thm:main1} concerns each single solution to
\eqref{eq:EVI}; when the $\lambda$-gradient flow $\mtS_t$ of $\phi$
exists we have further
interesting properties, showing that the formulation by \eqref{eq:EVI}
is really {s}tronger than all the other metric approaches.
\subsection{$\lambda$-Gradient flows and $\lambda$-convexity along geodesics}

Let us first recall the notion of (minimal, constant speed) geodesics
in a metric space $X$ and the related convexity.

\begin{definition}[Constant speed geodesics] A curve $\gamma:[0,1]\freccia X$ is a \emph{constant speed geodesic} {(or simply \emph{geodesic})} if
\begin{equation}\label{geod}
 {\Dist{\gamma_s}{\gamma_t}}=|t-s|\,{\Dist{\gamma_0}{\gamma_1}}\quad\forevery 0\leq s\leq t\leq1.
\end{equation}
A set $D\subset X$ is geodesically convex if every couple of points
$x_0,x_1\in D$ can be connected by a geodesic $\gamma$ contained in $D$.
\end{definition}

\begin{definition}[$\lambda$-convexity along curves and geodesically
  $\lambda$-convex functionals]
  We say that $\phi:X\freccia(-\infty,+\infty]$ is
  \emph{$\lambda$-convex} along a curve $\gamma:[0,1]\to X$
  if
  \begin{equation}\label{lambdaconv}
    \phi(\gamma_s)\leq (1-s)\phi(\gamma_0)+s\phi(\gamma_1)-\frac{\lambda}{2}s(1-s)\sfd^2(\gamma_0,\gamma_1)\quad\forevery s\in[0,1].
  \end{equation}
  We say that $\phi$ is \emph{geodesically $\lambda$-convex}
  if every couple of point{s}
  $x_0,x_1\in D(\phi)$ can be connected by a geodesic $\gamma$ along
  which $\phi$ is $\lambda$-convex.
  If $\phi$ is geodesically convex and it is $\lambda$-convex along
  \emph{every} geodesic connecting $x_0,x_1\in D(\phi)$ in
  $\overline{D(\phi)}$ then we say that $\phi$ is \emph{strongly}
  geodesically $\lambda$-convex.
\end{definition}

\begin{theorem}[\cite{Daneri-Savare08}]
  \label{eviconv}
  If the $\lambda$-gradient flow $\mtS_t$ of $\phi$ exists then $\phi$
  is $\lambda$-convex along any geodesic in $\overline{D(\phi)}$.
  In particular, if $\overline{D(\phi)}$ is geodesically convex, then
  $\phi$ is strongly geodesically $\lambda$-convex.
\end{theorem} 
\begin{proof} Let $\gamma:s\in [0,1]\mapsto\gamma^s\in \overline{D(\phi)}$ be a geodesic
    with $\gamma^0,\gamma^1\in D(\phi)$ and let us set $\gamma^s_t:=\mtS_t(\gamma^s)$.

    Applying \eqref{eq:cap1:78bis} we have for every $s\in [0,1]$ and $t>0$
    \begin{align}
      &\frac{1}{2}\rme^{\lambda t
      }\sfd^2(\gamma^s_t,\gamma^0)-\frac{1}{2}\sfd^2(\gamma^s,\gamma^0)\leq
      \mtE_{\lambda}(t)\bigl(\phi(\gamma^0)-\phi(\gamma^s_t)\bigr),\label{intevi0}\\
      &\frac{1}{2}\rme^{\lambda t
      }\sfd^2(\gamma^s_t,\gamma^1)-\frac{1}{2}\sfd^2(\gamma^s,\gamma^1)\leq
      \mtE_{\lambda}(t)\bigl(\phi(\gamma^1)-\phi(\gamma^s_t)\bigr){.}\label{intevi1}
    \end{align} 
    Multiplying \eqref{intevi0} by $(1-s)$ and \eqref{intevi1} by $s$ and adding the two inequalities we get
    \begin{align}\label{sumineq}
      \frac{{\rme}^{\lambda t}}{2}\bigl((1-s)\sfd^2(\gamma^s_t,\gamma^0)+s \sfd^2(\gamma^s_t,\gamma^1)\bigr) &-\frac{1}{2}\bigl((1-s)\sfd^2(\gamma^s,\gamma^0)+s\sfd^2(\gamma^s,\gamma^1)\bigr)\notag\\
      &\leq\mtE_{\lambda}(t)\bigl((1-s)\phi(\gamma^0)+s\phi(\gamma^1)-\phi(\gamma^s_t)\bigr).
    \end{align}
    We now observe that the elementary inequality 
    \begin{align}\label{elineq}
      (1-s)a^2+sb^2\geq s(1-s)(a+b)^2\quad\forevery a,b\in\R,\quad s\in[0,1],
    \end{align}
    and the triangular inequality yield
    \begin{align}\label{3.15}
      (1-s)\sfd^2(\gamma^s_t,\gamma^0)+s \sfd^2(\gamma^s_t,\gamma^1)&\overset{\eqref{elineq}}{\geq} s (1-s)\bigl({\Dist{\gamma^s_t}{\gamma^0}}+{\Dist{\gamma_t^s}{\gamma^1}}\bigr)^2\notag\\
      &\geq s(1-s) \sfd^2(\gamma^0,\gamma^1).
    \end{align}
    On the other hand, since $\gamma$ is a geodesic we have
    \begin{equation}\label{segmeq}
      (1-s)\sfd^2(\gamma^s,\gamma^0)+s\sfd^2(\gamma^s,\gamma^1)=s(1-s)\sfd^2(\gamma^0,\gamma^1).
    \end{equation}
    Inserting \eqref{3.15} and \eqref{segmeq} in \eqref{sumineq} we get
    \begin{align}\label{3.17}
      \frac{{\rme}^{\lambda t}-1}{2}s(1-s)\sfd^2(\gamma^0,\gamma^1)\leq\mtE_{\lambda}(t)\bigl((1-s)\phi(\gamma^0)+s\phi(\gamma^1)-\phi(\gamma^s_t)\bigr).
    \end{align}
    Dividing then both sides of \eqref{3.17} by $\mtE_{\lambda}(t)$ and passing to the limit as $t\downarrow0$ we obtain
    \begin{equation*}
      \phi(\gamma^s)\leq(1-s)\phi(\gamma^0)+s\phi(\gamma^1)-\frac{\lambda}{2}s(1-s)\sfd^2(\gamma^0,\gamma^1)\quad\forevery s\in[0,1].
    \end{equation*} 
  \end{proof}
  
\subsection{$\lambda$-gradient flows and curves of maximal slope}
\begin{definition}[Curves of maximal slope]
  \label{def:maxslope}
  We say that a curve $u\in
  AC^2_{\mathrm{loc}}(0,+\infty;X)$ is a curve of maximal slope for the
  functional $\phi$ if 
  the energy dissipation inequality 
\begin{equation}\label{cmaxs}
 \frac{1}{2}\int_s^t|\u|^2(r)\,dr+\frac{1}{2}\int_s^t|\partial\phi|^2(u_r)\,dr\le
 \phi(u_s)-\phi(u_t)
\end{equation}
holds for all $0<s\leq t<+\infty$.
\end{definition}
The notion of curve of maximal slope has been first introduced (in a slightly different form) by De Giorgi and
provides a weak notion of gradient flow for nonsmooth functionals, also nonconvex.
If $\phi$ admits a $\lambda$-gradient flow according to {D}efinition \ref{def:GFlow}, then these two definitions
coincide.
\begin{theorem}
  \label{thm:MaxSlopeAreGFlows}
  Let us assume that
  the $\lambda$-gradient flow $\FlowName$ of $\phi$ exists
  and let
  $u\in AC^2_{\rm loc}(0,+\infty;X)$ be satisfying
  \eqref{cmaxs} with $\lim_{t\downarrow0}u_t=u_0\in D(\phi)$.
  Then $u_t=\Flow t{u_0}$ for every $t\ge0$ and \eqref{cmaxs} is in
  fact an identity for every $0\le s<t<+\infty$.
\end{theorem}

\subsection{$\lambda$-gradient flows and the minimizing movement{s}
  variational scheme}
A general variational method to approximate gradient flows (and often to prove their existence) is 
provided by the so-called
\emph{minimizing movements} variational scheme.
In his original formulation (see e.g.\ \cite{DeGiorgi93}), the method consists in finding
a discrete approximation $\Pc U\tau$ of the continuous
gradient flow $u$ by solving a recursive
variational scheme, which is the natural generalization of
\eqref{eq:21}
to a metric-space setting.
If $\tau>0$ denotes the step size of the uniform partition $\cP_\tau$
\eqref{eq:19},
starting from a suitable approximation $U^0_\tau$ of $u_0$ one looks
at each step $((n-1)\tau,n\tau]$ for the minimizers of the functional
\begin{equation}
  \label{eq:cap4:30}
  U\mapsto \Phi(\tau,U^{n-1}_\tau;U):=\frac 1{2\tau}\DistSquare{U}{U^{n-1}_\tau}+\phi(U).
\end{equation}
$\Pc U\tau$ thus takes a value
$U^n_\tau\in \argmin \Phi(\tau,U^{n-1};\cdot)$ on each interval $((n-1)\tau,n\tau]$.
\begin{definition}[The minimizing movement variational scheme]
  \label{def:MMS}
  Let us consider a time step $\tau>0$
  and a discrete initial datum $U^0_\tau\in \Dom\phi $.
  A $\tau$-discrete minimizing movement starting from $U^0_\tau$
  is any sequence $(U^n_\tau)_{n\in \N}$ in $\Dom\phi$ which satisf{ies}
   \begin{equation}
     \label{eq:74}
     \Phi(\tau,U^{n-1}_\tau;U^{n}_\tau)\le
     \Phi(\tau,U^{n-1}_\tau;V)\quad \forevery V\in X,\ n\in \N.
   \end{equation}
   A discrete solution $\Pc U\tau$ is any piec{e}wise constant
   interpolant of a $\tau$-discrete minimizing movement on the grid
   $\cP_\tau$
   defined by
   \begin{equation}
     \label{eq:55}
     \overline{U}_{\ttau}(0)=U^0_{\ttau},\quad \overline{U}_{\ttau}(t)\equiv\unt\quad\text{if $t\in(t^{n-1}_{\ttau},t^n_{\ttau}]$, $n\geq1$.}
   \end{equation}
\end{definition}
The existence of a minimizing sequence $\{U^n_\tau\}_{n\in\N}$
is usually obtained by invoking the direct method of the
Calculus of Variations,
thus requiring that the functional
\eqref{eq:cap4:30}
has compact
sublevels with respect to some
Hausdorff topology $\sigma$ on $X$
(see e.g.\ the setting of \cite[\S\, 2.1]{Ambrosio-Gigli-Savare05}).
In the next section we will discuss another possibility, still considered in
\cite{Ambrosio-Gigli-Savare05},
when the functional
\eqref{eq:cap4:30} satisfies a strong convexity assumption.

In a general setting it is also possible to avoid these restrictions by applying
the Ekeland's Variational Principle to
the functional \eqref{eq:cap4:30}:
this approach only requires the completeness of the metric space.
\begin{definition}[A relaxed minimizing movement variational scheme]
  \label{def:rMMS}
  Let us consider a time step $\tau>0$, a relaxation parameter $\eta\ge0$,
  and a discrete initial datum $U^0_\taueta\in \Dom\phi $.
  A $(\tau,\eta)$-discrete minimizing movement starting from $U^0_\taueta$
  is any sequence $(U^n_\taueta)_{n\in \N}$ in $\Dom\phi$ which satisf{ies}
  \begin{subequations}
    \label{problemb}
    \begin{equation}
      \label{eq:cap1:67a}
      \begin{aligned}
        \Phi(\tau,U^{n-1}_\taueta;U^{n}_\taueta)\le
        \Phi(\tau,U^{n-1}_\taueta;V) +
        \frac \eta2\,\Dist{U^n_\taueta}{U^{n-1}_\taueta}\,\Dist
        V{U^n_\taueta}\quad \text{for every }V\in \Dom\phi ,
      \end{aligned}
    \end{equation}
    and the further condition
    \begin{equation}
      \label{eq:cap4:22a}
      \Phi(\tau,U^{n-1}_\taueta;U^n_\taueta)=\frac1{2\tau}\DistSquare{U^{n}_\taueta}{U^{n-1}_\taueta}+
      \phi(U^n_\taueta)\le
      \phi(U^{n-1}_\taueta),
    \end{equation}
  \end{subequations}
  for every $n\in \N$.
  A $(\tau,\eta)$-discrete solution $\Pc U\taueta$ is any piecewise
  constant interpolant of a $(\tau,\eta)$-discrete minimizing movement
  on the grid $\cP_\tau$, as in \eqref{eq:55}.
\end{definition}
Notice that when $\eta=0$ a solution to \eqref{eq:cap1:67a} is a
minimizer of \eqref{eq:cap4:30} (and in particular satisfies
\eqref{eq:cap4:22a}), so that the usual discrete solutions arising
from the minimizing movement scheme are included in this more general
relaxed framework.
The next result \cite{Savare10}, which follows directly from Ekeland's
variational principle, shows that the previous scheme
is always solvable when $\eta>0$.
\begin{theorem}
  \label{thm:RMMSexist}
  Let us assume that $X$ is complete and $\phi$ is quadratically
  bounded from below, i.e.\ for some $\kappa_o,\phi_o\in \R${, $o\in X$} 
  \begin{equation}
    \label{eq:30a}
    \phi(x)+\frac{\kappa_o}2\sfd^2(x,o)\ge \phi_o\quad\forevery x\in X.
  \end{equation}
  Then for every $\eta>0$, $\tau>0$ with $\tau^{-1}>-\kappa_0$, and $U^0_\taueta\in
  D(\phi)$, the relaxed minim{i}zing movement scheme admits at least a
  $(\tau,\eta)$-discrete solution $(U^n_\taueta)_{n\in \N}$.
\end{theorem}
Since under the general assumptions of Theorem \ref{thm:RMMSexist}
the relaxed minimizing movement scheme admits a $(\tau,\eta)$-discrete
solution $\Pc U\taueta$ for fixed $\eta>0$ and
arbitrarily small step size $\tau$,
it is natural to ask what its limit as $\tau\downarrow0$.
A first result in this direction is provided by the next theorem,
which shows that the minimizing movement scheme is \emph{consistent}
with the definition of $\lambda$-gradient flow \ref{def:GFlow}.
Notice that in Theorem \ref{thm:Error_Estimate}
we will assume \emph{a priori} that the $\lambda$-gradient
flow of $\phi$ exists to get the convergence of $\Pc U\taueta$; in
Section \ref{sect:genteo} we will discuss how to remove this strong
assumption.
Still it is sometimes useful to know that any $\lambda$-gradient flow,
no matter how it has been constructed, admits a uniformly
converging discrete approximation, which exhibits nice variational properties.
\begin{theorem}
  \label{thm:Error_Estimate}
  Let us assume that there exists
  the $\lambda$-gradient flow $\FlowName_t$ of $\phi$ according to
  Definition
  \ref{def:GFlow} and that
  $\overline{D(\phi)}$ is geodesically convex.
  Let $\tau>0,\eta\ge0$ satisfy $\eta-\lambda<\frac 1{2\tau}$,
  and let the sequence $(U^n_\taueta)_{n\in \N}\subset D(\phi)$ be a $(\tau,\eta)$-discrete
  minimizing movement with
  $U^0_\taueta\in \DomainSlope\phi $.
  Setting $\alpha=\alpha_\taueta:=\frac 1{2\tau}\log(1+2(\lambda-\eta)\tau)$ we have
  the \emph{a priori} error estimate
  \begin{equation}
    \label{eq:24bis}
    \Dist{\mtS_t(u_0)}{\Pc U\taueta(t)}\le \Dist{u_0}{U^0_\taueta}+
    \rme^{-\alpha T}\sqrt {T\tau}\MetricSlope\phi{U^0_\taueta}\quad
    \forevery t\in [0,T].
  \end{equation}
  In particular if for some $\eta\ge0$ and every $\tau\in (0,\tau_0)$ $\Pc
  U\taueta$
  is a family of $(\tau,\eta)$-discrete solutions with $U^0_\taueta=u_0\in D(|\partial\phi|)$, then
  $\lim_{\tau\downarrow0}\Pc U\taueta(t)=\mtS_t(u_0)$ uniformly on
  every compact {interval}.
\end{theorem}
Let us remark that $\eta$ has been kept fixed in the previous
convergence result, so that the coefficients $\alpha_\taueta=\frac
1{2\tau}\log(1+2(\lambda-\eta)\tau)$
in the estimate \eqref{eq:24bis} are
uniformly bounded from below as $\tau\down0$.

\subsection{Stability of $\lambda$-gradient flows under
  $\Gamma$-convergence}
\label{subsec:stability}
We conclude this section by showing a simple stability property of Gradient
Flows with respect to perturbations of the generating functional
$\phi$.
Here we consider a coercive family
of $\Gamma$-converging functionals $\phi^h:X\to(-\infty,+\infty]$, $h\in \bar\N=\N\cup
\{+\infty\}$, which are quadratically bounded from below, uniformly
w.r.t.\ $h$: for some $o\in X$, $\phi_o,\kappa_o\in \R$
they satisfy
\begin{equation}
  \label{eq:cap1:13bis}
  \phi^h(x)+\frac {\kappa_o}{2}\DistSquare x{o }\ge \phi_o\quad
  \text{for every } x\in  X,\ h\in \N.
\end{equation}
In the next definition we jointly recall the (sequential) notions of $\Gamma$-convergence and of coercivity
\cite[Def. 1.12]{DalMaso93}.
For notational convenience, we will identify monotone subsequences $(h_n)_{n\in \N}$
with their unbounded image $H=\{h_n:n\in \N\}\subset \N$;
expressions like $\lim_{h\in H}$, $\liminf_{h\in H}$ have an obvious meaning as limits for 
$h\uparrow +\infty, h\in H$.
\begin{definition}[Sequential $\Gamma( X,\DistName)$-convergence of
  coercive functionals]
  \label{def:Gamma-coercive}
  We say that $(\phi^h)_{h\in \N}$ is a coercive family of functionals
  $\Gamma( X,\DistName)$-sequentially converging to a proper functional $\phi^\infty: X
  \to (-\infty,+\infty]$
  if 
  the following two conditions are satisfied:
  \begin{enumerate}
  \item For every infinite subset $H\subset \N$ and
    every bounded sequence $(x^h)_{h\in H}$
    with $\sup_{h\in H} \phi^h(x^h)<+\infty$, there exists
    an infinite subsequence $H'\subset H$ such that 
    $\lim_{h\in H'}x^h=\xinfty\in \Dom\phinfty$ and
    \begin{equation}
      \label{eq:cap21:3a}
      \liminf_{h\in H'}\phi^{h}(x^{h})\ge \phinfty(\xinfty).
    \end{equation}
  \item For every $\barxinfty\in D(\phinfty)$ there exists a sequence
    $(\bar x^h)_{h\in \N}$ such that 
    \begin{equation}
      \label{eq:cap21:2a}
      \lim_{h\up+\infty}\Dist{\bar x^h}\barxinfty=0,\quad
      \lim_{h\up+\infty}\phi^h(\bar x^h)=\phinfty(\barxinfty).
    \end{equation}
  \end{enumerate}
\end{definition}
It is possible to prove that 
a coercive family of $\lambda$-convex functionals $(\phi^h)_{h\in \N}$
$\Gamma( X,\DistName)$-converging to $\phi^\infty$
always satisfies the uniform lower bound \eqref{eq:cap1:13bis}.

Let us now state our first convergence result \cite{Savare10}.
\begin{theorem}
  \label{thm:main_stability1}
  Let $(\phi^h)_{h\in \N}$ be a coercive family of functionals
  $\Gamma( X,\DistName)$-converging to $\phi^\infty$ and
  let
  us assume that the
  $\lambda$-gradient flows $\FlowName^h$
  exist 
  for every $h\in \N$.
  Then the functional $\phinfty$ admits
  a $\lambda$-gradient flow $\FlowName^\infty$ 
  and for every sequence $u^h_0\in
  \overline{D(\phi^h)}$ converging to $u_0^\infty\in \overline{\Dom\phinfty }$ we have
  \begin{gather}
    \label{eq:cap21:4}
    \lim_{h\up+\infty}\FlowName^h_t(u^h_0)=\FlowName^\infty_t(u_0^\infty),\quad
    \lim_{h\up+\infty}\phi^h(\FlowName^h_t(u^h_0) )=\phinfty(\FlowName^\infty_t(u_0^\infty) )\quad
    \forevery t>0,
  \end{gather}
  {locally uniformly on $(0,+\infty).$}
\end{theorem}
\begin{proof}
  Here we consider the simpler case when \eqref{eq:cap1:13bis} holds
  for $\kappa_o=0$; it is not restrictive to assume $\lambda\le 0$ and $\phi_o\ge0$.

  \emph{Step 1: uniform bounds.}
  We set $u^h_t:=\FlowName_t^h(u_0^h)$ and we
  fix a compact time interval $[0,T]$, $T>0$,
  a point $o^\infty\in \Dom\phinfty$ and a corresponding sequence $o^h$ as in \eqref{eq:cap21:2a}.
  \eqref{eq:cap1:78bis} yields
   \begin{gather}
      \label{eq:cap2:12}
        \DistSquare{u_t^h}{o^h}
        \le \Big(\DistSquare{u_0{^h}}{o^h}+
        2\sfE{\lambda}t \,\phi^h(o^h)
        \Big)\rme^{-\lambda t};
    \end{gather}
    and therefore
    there exists a constant $C_1(T)$ independent of $h$ such that
    \begin{equation}
      \label{eq:78}
      \Dist{u^h_t}{{o^h}}\le C_1(T)\quad\forevery t\in [0,T],\ h\in \N.
    \end{equation}
    The regularizing estimate
    \eqref{eq:2bis} yields
    \begin{equation}
      \label{eq:9bis}
      \frac {\rme^{\lambda t}}2\DistSquare{u^h_t}{o^h}+
      \sfE\lambda t \phi^h(u^h_t)+\frac {\big(\sfE \lambda t\big)^2}2\MetricSlopeSquare{\phi^h}{u^h_t}\le
      \frac 12\DistSquare{u^h_0}{o^h}+\sfE\lambda t\phi^h(o^h)
      \le C_2(T)
    \end{equation}
  if $t\in (0,T]$, for a suitable constant $C_2(T)$ independent of $h$.

  In particular for every $0<S<T$ there exists a constant $C(S,T)$
    such that
    \begin{equation}
      \label{eq:77}
      \phi^h(u^h_t)\le C(S,T),\quad \MetricSlope{\phi^h}{u^h_t}=|\dot u^h_{t+}|\le
      C(S,T)\quad
      \forevery t\in [S,T].
    \end{equation}
    \emph{Step 2: compactness.}
    By the estimates of the previous point, the sequence $(u^h)_{h\in \N}$
    is uniformly Lipschitz in each bounded interval $[S,T]$ of
    $(0,+\infty)$ and for every fixed $t$
    $\{u^h_t\}_{h\in \N}$ satisfies the assumptions
    of Definition \ref{def:Gamma-coercive}, so that $(u^h_t)_{h\in
      \N}$ is relatively compact in $ X$.
    Applying Ascoli-Arzel\`a theorem we can find a subsequence
    $H=(h_n)_{n\in \N}$ such that $u^{h_n}$ converge locally uniformly
    in time
    to a locally Lipschitz curve ${u^\infty}$ in $(0,+\infty)$.

    \emph{Step 3: characterization of the limit.}
    Let us now fix an arbitrary point $v^\infty\in \Dom\phinfty$ and a
    corresponding approximating sequence $v^h\in \Dom{\phi^h}$ as in \eqref{eq:cap21:2a}.
    By \eqref{eq:cap1:78bis} of Theorem \ref{thm:uniqueness} we know
    that
    \begin{equation}
      \label{eq:39}
      \frac {\rme^{\lambda(t-s)}}2\DistSquare{u^h_t}{v^h}-
      \frac 12\DistSquare{u^h_s}{v^h}\le
      \sfE\lambda{t-s}\Big(\phi^h({v^h})-\phi^h(u^h_t)\Big);
    \end{equation}
  We then pass to the limit in \eqref{eq:39} as $h\up+\infty, h\in H$,
  using the facts that $u^h_t$ converges pointwise to $u_t$ in $ X$
  and applying \eqref{eq:cap21:3a} for $u^h_t$ and \eqref{eq:cap21:2a} for $v^h$;  
  we obtain
  \begin{equation}
    \label{eq:75}
    \frac {\rme^{\lambda(t-s)}}2\DistSquare{u^\infty_t}{v^\infty}-
    \frac 12\DistSquare{u_s^\infty}{v^\infty}\le
    \sfE\lambda{t-s}\Big(\phi(v^\infty)-\phi(u_t^\infty)\Big)
  \end{equation}
  for every $v^\infty\in \Dom\phinfty$, $0\le s<t$.
  A further application of Theorem  \ref{thm:uniqueness}
  shows that $u^\infty$ solves $\EVI X\DistName\phinfty\lambda$.

  In order to check that $\lim_{t\downarrow0}u^\infty_t=u^\infty_0$
  we use \eqref{eq:75} at $s=0$ and the lower semicontinuity of $\phinfty$,
  which yields
  \begin{equation}
    \label{eq:38}
    \limsup_{t\downarrow 0}\DistSquare{u^\infty_t}{v^\infty}\le \DistSquare{u^\infty_0}{v^\infty}\quad
    \text{for every }v^\infty\in \Dom\phinfty;
  \end{equation}
  since $u_0^\infty\in \overline{\Dom\phinfty}$
  we conclude that $\lim_{t\down0}\Dist{u_t^\infty}{u_0^\infty}=0.$

  Since the limit is the unique solution of
  $\EVI X\DistName\phinfty\lambda$ starting from
  $u_0^\infty$, we conclude that the whole sequence $u^h$ converge to
  $u^\infty$.

  \emph{Step 4: convergence of energy.}
  We argue as in the proof of \eqref{eq:cap2:16} and \eqref{eq:80}:
  for a fixed $t>0$ and applying \eqref{eq:cap21:2a} to $u^\infty_t$
  we find a sequence $(\bar u^h_t)_{h\in\N}$ converging to
  $u^\infty_t$ with $\lim_{h\up+\infty}\phi^h(\bar u^h_t)=\phinfty(u^\infty_t)$.
  By estimate \eqref{eq:9bis}, the slope $\MetricSlope{\phi^h}{u^h_t}$ is
  uniformly bounded by a constant $M_t$ so that
  \begin{displaymath}
    \phi^h(\bar u^h_t)\ge \phi^h(u^h_t)-M_t\Dist{\bar
      u^h_t}{u^h_t}-\frac\lambda 2\DistSquare{\bar
      u^h_t}{u^h_t}.
  \end{displaymath}
  Passing to the limit as $h\up+\infty$ we get 
    $\limsup_{h\up+\infty}\phi^h(u^h_t)\le \phinfty(u^\infty_t)$,
    which combined with \eqref{eq:cap21:3a} yields the second identity of \eqref{eq:cap21:4}.
\end{proof}

\end{section}

\begin{section}{Convergence of the minimizing movement method and generation results}\label{sect:genteo}

We have seen in Theorem \ref{thm:Error_Estimate} a first convergence
theorem for the relaxed minimizing movement method: it basically says
that \emph{if $\phi$ admits a $\lambda$-gradient flow} according to
Definition \ref{def:GFlow} then any family of discrete solutions
converges to the unique continuous solution of \eqref{eq:EVI} as the
time step converges to $0$.

In this section we revert this point of view and we try to prove
\emph{the existence of the $\lambda$-gradient flow} when $\phi$
is geodesically $\lambda$-convex  by studying the
convergence of the (relaxed) minimizing movement Method.

In the following we present three different results in this direction{:}
\begin{enumerate}
\item {A} simpler convergence result when \emph{the sublevels of $\phi$
    are locally compact}: in this case we avoid any
  geometric restriction on the distance $\sfd$ of $X$ and we do not
  need any Cauchy estimate. On the other hand, the (not necessarily
  unique) limit points of the discrete solutions are just curves of
  Maximal Slope, according to Definition \ref{def:maxslope}: in
  general {it} is not possible to prove that they solve \eqref{eq:EVI}.
\item A first generation result for $\lambda$-gradient flows, by
  assuming
  that the minimizing movement generating functional $\Phi(\tau,U;V)$ defined by
  \eqref{eq:cap4:30} satisfies a suitable convexity property
  (which results from the combination of the convexity of $\sfd^2$ and
  of $\phi$).
\item A second generation result when $\sfd^2(\cdot,v)$ is semiconcave
  along geodesics and the metric space satisfies a \emph{local angle
    condition} between triple of geodesic{s} emanating from the same point.
\end{enumerate}
Differently from the first approach, the last ones provide explicit
Cauchy estimates ensuring the convergence of the method and do not require any
local compactness of the sublevels of $\phi$.

\subsection{Convergence of the variational scheme in the locally
  compact case}
\label{subsec:compact}
Let us first consider the case when $\phi$ is geodesically
$\lambda$-convex and its sublevels are locally compact, i.e.{$\exists\,o\in X$ s.t.}\
\begin{equation}
  \label{eq:30}
  \Big\{x\in X:\phi(x)\le R \text{ and }\sfd(x,o)\le
  R\Big\}\quad\text{are compact in $X$ }\forevery R>0.
\end{equation}
Combining \cite[Proposition 2.2.3, Corollary
2.4.11]{Ambrosio-Gigli-Savare05} we get
\begin{theorem}[Limit{s} of discrete minimizing movements are curves of
  maximal slope]
  If $\phi$ is geodesically $\lambda$-convex and satisfies
  \eqref{eq:30} then for every $\tau>0$ satisfying
  $\tau^{-1}>-\lambda$ and $U^0_\tau\in D(\phi)$ the minimizing
  movement variational scheme admits at least one solution
  $(U^n_\tau)_{n\in \N}$.
  If moreover
  \begin{equation}
    \label{eq:32}
    \lim_{\tau\downarrow0}U^0_\tau=u_0,\quad\lim_{{\tau\downarrow0}}\phi(U^0_\tau)=\phi(u_0),
  \end{equation}
  and $\Pc U\tau$ is a family of discrete solutions, any infinitesimal
  sequence of time steps $\tau_n\down0$ admits a convergent
  subsequence (still denoted by $\tau_n$) and a limit curve $u\in
  AC^2_{\rm loc}([0,+\infty);X)$ such that
  \begin{equation}
    \label{eq:31}
    \lim_{n\up+\infty}\Pc U{\tau_n}(t)=u_t,\quad
    \lim_{n\up+\infty}\phi(\Pc
    U{\tau_n}(t))=\phi(u_t)\quad\forevery t\ge0
  \end{equation}
  uniformly in each compact interval $[0,T]$.
  $u$ is a curve of maximal slope (see {D}efinition \ref{def:maxslope}), satisfying the energy identity
  \begin{equation}\label{cmaxs2}
    \frac{1}{2}\int_s^t|\u|^2(r)\,dr+\frac{1}{2}\int_s^t|\partial\phi|^2(u_r)\,dr=
    \phi(u_s)-\phi(u_t)\quad
    \text{for all $0<s\leq t<+\infty$.}
\end{equation}
\end{theorem}
\begin{corollary}[Existence of Curves of Maximal Slope]
  Under the same assumption{s} of the previous theorem, for every $u_0\in
  D(\phi)$ there exists a curve of maximal slope $u\in AC^2_{\rm loc}([0,+\infty);X)$
  starting from $u_0$ and satisfying \eqref{cmaxs2}.  
\end{corollary}

\begin{subsection}{Generation of $\lambda$-gradient flows by strong
    convexity of $\Phi$.}

In the case when $X$ is an Hilbert space and $\phi$ is a l.s.c.\ $\lambda$-convex functional,
it is well known that the minimizing movement variational scheme admits a unique solution and
the corresponding discrete solution $U_\tau$ converges to the solution of
\eqref{eq:EVI}. Applying this approximation scheme, it is then possible to show the existence
of the $\lambda$-gradient flow of $\phi$ according to {D}efinition {\ref{def:GFlow}}.

Similar results for minimizing movements of convex functionals in Banach spaces do not always hold: indeed, the characterization of gradient flows through the EVI depends not only on the convexity of $\phi$ but also on structural properties of the distance $\sfd$.    

\noindent One fundamental property is the $1$-convexity of the function $v\mapsto\frac{1}{2}\sfd(v,w)^2$, i.e.
\begin{align}\label{1conv}
 \sfd^2(v_s,w)\leq&(1-s)\sfd^2(v_0,w)+s\sfd^2(v_1,w)-s(1-s)\sfd^2(v_0,v_1)\notag\\
&\forevery v_0,v_1, \,\forevery [0,1]\ni s\mapsto v_s\text{ geodesic between $v_0$ and $v_1$},
\end{align}
which in Banach spaces is equivalent to the fact that $\sfd$ is induced by a scalar product. 

\noindent \eqref{1conv} is satisfied by the geodesic distance on Riemannian manifolds of non-positive sectional curvature and characterizes the Aleksand{ro}v \emph{non-positively curved} (NPC) length spaces,
see e.g.\ \cite{Jost97,Burago-Burago-Ivanov01}.

Actually, using \eqref{1conv} and adapting a Crandall-Liggett argument, {\textsc{Mayer}} \cite{Mayer98} was able to prove \eqref{contrh} and then \eqref{exp} also for geodesically convex functionals on NPC spaces.

A crucial consequence of \eqref{1conv} and the $\lambda$-convexity of $\phi$ is that
the generating functional
${\Phi(\tau,V;U)}$ of the minimizing movement scheme \eqref{eq:cap4:30}
\begin{equation}
  \label{eq:33}
  \Phi(\tau, V;U):=\frac 1{2\tau}\sfd^2(U,V)+\phi(U)\quad
  \tau>0,\ U,V\in X,
\end{equation}
satisfies the $\tau^{-1}+\lambda$-convexity condition along geodesics, i.e.
\begin{equation}
  \label{eq:34}
  \text{the map $U\mapsto \Phi(\tau,{V;U})$ is geodesically $(\tau^{-1}+\lambda)$-convex for every $V\in X$.}
\end{equation}
One of the main contribution{s} of \cite[Chapter 4]{Ambrosio-Gigli-Savare05}
is to show that \eqref{eq:34} can be relaxed,
by assuming the $(\tau^{-1}+\lambda)$-convexity of $\Phi(\tau,{V;\cdot})$ along more general families
of curves in $X$ connecting two arbitrary points in $D(\phi$).

This improvement has been essential to apply the generation result in Wasserstein spaces,
which do not satisfy \eqref{1conv} except for the $1$-dimensional case.

\begin{theorem}[Convergence of the minimizing movement scheme and generation result \cite{Ambrosio-Gigli-Savare05}] \label{genteo}
Let us assume that the functional $\Phi$ defined in \eqref{eq:33} satisfies the following property: 
$\forevery V,U_0,U_1\in D(\phi)$ there exists a curve $\gamma_s:[0,1]\freccia X$ with $\gamma_0=U_0$ and $\gamma_1=U_1$, such that  
\begin{equation}\label{convpphi}
  U\mapsto\Phi(\ttau,V;U)\text{ is $\biggl(\frac{1}{\ttau}+\lambda\biggr)$-convex on $\gamma$ for each $0<\ttau<\frac{1}{\lambda^-}$,}
\end{equation}
i.e.
\begin{equation}
  \label{eq:33bis}
  \Phi(\tau,V;\gamma_s)\le
  (1-s)\Phi(\tau,V;U_0)+s\Phi(\tau,V;U_1)-\frac {1+\lambda\tau}{2\tau}s(1-s)\sfd^2(U_0,U_1).
\end{equation}
\begin{enumerate}\item 
  For every $U^0_\tau=u_0\in \overline{D(\phi)}$ and $\tau>0$
  with $1+\tau\lambda>0$ the minimizing movement method \ref{def:MMS}
  admits a unique solution $(U^n_\tau)_{n\in \N}\subset D(\phi)$
  \item The corresponding discrete solutions $\bar U_\tau$ converge to $u$
  as $\tau\downarrow0$ uniformly on compact intervals.
\item The limit $u$
  is the unique solution of \eqref{eq:EVI}.
  In particular, $\phi$
  admits a $\lambda$-gradient flow according to Definition
  \ref{def:GFlow}, thus satisfying all the properties stated in
  Theorem \ref{thm:main1}.
\item
  There exist universal constants $C_{\lambda,T}$ such that if $u_0\in
  D(|\partial\phi|)$ the optimal error estimate holds:
  \begin{equation}
    \label{eq:34bis}
    \sfd(u(t),\bar U_\tau(t))\le C_{\lambda,
      T}|\partial\phi|(u_0)\, \tau\quad\forevery t\in [0,T].
  \end{equation}
\end{enumerate}
\end{theorem}
\renewcommand{\eeta}{\eta}
\noindent The main arguments of the proof
of Theorem \ref{genteo} in a simplified setting can be found in
\cite{Savare04}.
Sub-optimal convergence estimates, inspired by the Crandall-Ligget
approach,
have also be{en} obtained
{in} a different way in \cite{Ambrosio-Savare-Zambotti09}
and in \cite{Clement09}.

\end{subsection}

\begin{subsection}{Generation results for geodesically convex
    functionals in spaces with a semiconcave squared distance}

  In this subsection we consider a geodesically $\lambda$-convex
  functional in a complete metric space $(X,\sfd)$ whose squared distance
  satisfies a \emph{semi-concavity} condition.
  
 \begin{definition}[Semi-concavity of the squared distance
   function]\label{defi:ksc}
   We say that $D(\phi)\subset X$ is a $\mathrm{K}$-SC (Semi-Concave) space if for every geodesic $[0,1]\ni s\mapsto v_s\in D(\phi)$ and for every $w\in D(\phi)$ we have
\begin{equation}
  \label{eq:35}
  \sfd^2(v_s,w)\geq(1-s)\sfd^2(v_0,w)+s\sfd^2(v_1,w)-\mathrm{K}s(1-s)\sfd^2(v_0,v_1)\quad\forevery s\in[0,1].
\end{equation}
\end{definition} 

\noindent
\textbf{Examples of $\mathrm{K}$-SC spaces}

\noindent -\textsc{PC spaces}: $X$ is positively curved (PC) in the Aleksandrov sense if and only if $X$ is $\mathrm{K}$-SC with $\mathrm{K}=1$.

\noindent -\textsc{Aleksandrov spaces}: if $X$ is an Aleksandrov space whose curvature is bounded from below by a negative constant $-k$ and $D=\mathrm{diam}(X)<+\infty$, then $X$ is a $\mathrm{K}$-SC space with $\mathrm{K}=\frac{D\sqrt{k}}{\tanh(D\sqrt{k})}$. This class includes all Riemannian manifolds whose sectional curvature is bounded from below.

\noindent -\textsc{Product and $\rL^2$-spaces}: if $\{(X_i,{\sfd}_i)\}_{i\in\N}$ is a countable collection of $\mathrm{K}$-SC spaces, then the product $\underset{i\in\N}{\prod}X_i$ with the usual product distance is a $\mathrm{K}$-SC space. If $\mu$ is a finite measure on some separable measure space $\Omega$, then $\XX:=\rL^2_{\mu}(\Omega;X)=\bigl\{f:\Omega\freccia X:\,\int_{\Omega}\sfd^2(f(\omega), x_0)\,\d\mu(\omega)<+\infty\text{ for some $x_0\in X$}\bigr\}$ endowed with the distance $d_{\XX}^2(f,g)=\int_{\Omega}\sfd^2(f(\omega),g(\omega))\,\d\mu(\omega)$ is $\mathrm{K}$-SC whenever $X$ is $\mathrm{K}$-SC.

\noindent -\textsc{Wasserstein space}: $(\PP_2(X), W_2)$ is
$\mathrm{K}$-SC if and only if $X$ is $\mathrm{K}$-SC (see the next section).

We will also assume that the (upper) angle between couple of geodesics
emanating from the same point satisfies a suitable condition.

\begin{definition}[Upper angles]Let $x^1$, $x^2$ be two geodesics emanating from the same initial point $x_0:=x^1_0=x_0^2$. Their \emph{upper angle} $\sphericalangle_u(x^1,x^2)\in[0,\pi]$ is defined by 
\begin{equation*}
 \cos(\sphericalangle_u(x^1,x^2)):=\underset{s,t\downarrow0}{\liminf}\,\frac{\sfd^2(x_0,x^1_s)+\sfd^2(x_0,x^2_t)-\sfd^2(x_s^1,x^2_t)}{2{\sfd}(x_0,x^1_s){\sfd}(x_0,x^2_t)}
\end{equation*}
 \end{definition}
\begin{definition}[Local angle condition (LAC)]\label{defi:lac} We say
  that $D(\phi)\subset X$ satisfies the \emph{local angle condition} (LAC) if
  for any triple of geodesics $x^i:[0,1]\to D(\phi)$, $i=1,2,3$,
  emanating from the same initial point $x_0$
  the corresponding angles $\theta^{ij}:=\sphericalangle_u(x^i,x^j)$ satisfy one of the following equivalent conditions:

\noindent \textrm{1. }$\theta^{12}+\theta^{23}+\theta^{31}\leq2\pi$.

\noindent \textrm{2. }There exist a Hilbert space $H$ and vectors $w^i\in H$ such that $\langle w^i,w^j\rangle_{H}=\cos(\theta^{ij})$ for $1\leq i,j\leq3$.

\noindent  \textrm{3. }For any choice of $\xi^1,\xi^2,\xi^3\geq0$ one has that $\overset{3}{\underset{i,j=1}{\sum}}\cos(\theta^{ij})\xi^i\xi^j\geq0$.
\end{definition}
\noindent
\textbf{Examples of {(LAC)} spaces}

\noindent -A \textsc{Banach space} $X$ satisfies (LAC) if and only if $X$ is a \textsc{Hilbert} space.

\noindent -\textsc{Riemannian manifolds and Aleksandrov spaces} with
curvature bounded from below satisfy (LAC). In particular if
\eqref{eq:35} holds with ${\mathrm{K}}=1$ then $X$ satisfies (LAC).

\noindent -\textsc{Product and $\rL^2$-spaces}: $\underset{i\in\N}{\prod}X_i$ satisfies (LAC) if and only if each $(X_i,{\sfd}_i)$ does; $\rL^2_{\mu}(\Omega;X)$ satisfies (LAC) if and only if $X$ satisfies it.

\noindent -\textsc{Wasserstein space}: $\PP_2(X)$ satisfies (LAC) if and only if $X$ does.

\begin{theorem}[Generation theorem for geodesically $\lambda$-convex
  functionals in $\mathrm{K}$-SC and $(LAC)$
  spaces]\label{genteofinal}
  Let $(X,\sfd)$ be a complete metric space and let
  $\phi:(-\infty,+\infty]$ be a proper, l.s.c. and
  $\lambda$-geodesically convex functional.
  \begin{enumerate}
  \item For every $\tau,\eta>0$ with $1+\tau\lambda>0$ and
    $U^0_\taueta=u_0\in D(\phi)$ the
    relaxed minimizing movement scheme \emph{(\ref{problemb}a,b)}
    admits at least one solution
    $(U^n_\taueta)_{n\in \N}$.
    \item
      If $D(\phi)$ is a $\mathrm{K}$-SC space and satisfies the (LAC)
      then the discrete solution $\bar U_\taueta$ converges to $u$ as
      $\tau\downarrow0$ uniformly in each compact interval.
    \item The limit $u$
      is the unique solution of \eqref{eq:EVI}.
      In particular, $\phi$
      admits a $\lambda$-gradient flow according to Definition
      \ref{def:GFlow}, thus satisfying all the properties stated in
      Theorem \ref{thm:main1}.
    \end{enumerate}

\end{theorem}


\end{subsection}

\end{section}

\section{Wasserstein spaces and diffusion
    equations}\label{sect:wassgf}
\begin{subsection}{The Wasserstein space}
Here, just to set the notation, we collect some basic definitions and properties of the Wasserstein space which will be used in the sequel. For a more detailed overview on this topic we refer to \cite{Villani03,Villani09}, \cite{Ambrosio-Gigli-Savare05} and \cite{Ambrosio-Savare06}.
\subsubsection*{Transport maps and couplings}
We denote by $X_i$, for some $i\in\N$, a separable and complete metric
space.
$\PP(X)$ is the space of Borel probability measures on $X$.

If $\mu\in\PP(X_1)$ and $\tt:X_1\freccia X_2$ is a Borel map, we denote by $\tt_{\#}\mu\in\PP(X_2)$ the \emph{push-forward of $\mu$ through {$\tt$}}, defined by
\begin{equation}\label{pushfor}
 \tt_{\#}\mu(B):=\mu(\tt^{-1}(B))\quad\forevery B\in \BB(X_2).
\end{equation}

We denote by $\pi^i$, for $i=1,...,n$, the canonical projection operator from a product space $X_1\times...\times X_n$ into $X_i$, defined by
\begin{equation*}
 \pi^i(x_1,...,x_n):=x_i.
\end{equation*}

Given $\mu_1\in\PP(X_1)$ and $\mu_2\in\PP(X_2)$, the class
$\Gamma(\mu_1,\mu_2)$ of \emph{transport plans} or \emph{couplings} between $\mu_1$ and $\mu_2$ is defined by
\begin{equation*}
 \Gamma(\mu_1,\mu_2):=\bigl\{\ggamma\in\PP(X_1\times X_2):\pi^1_{\#}\ggamma=\mu_1,\,\pi^2_{\#}\ggamma=\mu_2\bigr\}.
\end{equation*}

To each couple of measures $\mu_1\in\PP(X_1)$, $\mu_2=\tt_{\#}\mu_1\in\PP(X_2)$ linked by a Borel map $\tt:X_1\freccia X_2$ we can associate the coupling
\begin{equation}\label{indplan}
 \mmu:=(\Id_{X_1}\times \tt)_{\#}\mu_1\in\PP(X_1\times X_2),\quad\text{$\Id_{X_1}$ being the identity map on $X_1$}.
\end{equation}
If $\mmu$ is representable as in \eqref{indplan} we say that $\mmu$ is
induced by $\tt$ and $\tt$ is a \emph{transport map} between $\mu_1$
and $\mu_2$.
Each coupling $\mmu\in\Gamma(\mu_1,\mu_2)$ concentrated on a $\mmu$-measurable graph in $X_1\times X_2$ admits the representation \eqref{indplan} for some $\mu_1$-measurable map $\tt$, which therefore transports $\mu_1$ into $\mu_2$. 
 
\subsubsection*{Wasserstein distance}

Given a complete and separable metric space $(X,\sfd)$ we denote by
$\PP_2(X)$ the space of Borel probability measures with finite
quadratic moment: $\mu\in \PP(X)$ belongs to $\PP_2(X)$ iff
\begin{equation}
  \label{eq:40}
  \int_X \sfd^2(x,x_o)\,\d\mu(x)<+\infty\quad\text{for some (and thus
    any) point }x_o\in X.
\end{equation}
For every couple of measures $\mu,\nu\in\PP_2(X)$ we consider the Kantorovich problem for the cost $\sfd^2$ 
\begin{equation}\label{kant}
  W^2_2(\mu,\nu):=\min\biggl\{\int_{X\times X}\sfd^2(x,y)\,\d\ggamma(x,y):\,\ggamma\in \Gamma(\mu,\nu)\biggr\}.
\end{equation}
It is not difficult to check, by the direct method of calculus of
variations, that the minimum problem
\eqref{kant} admits at least a solution.
The subset of $\Gamma(\mu,\nu)$ given by the \emph{optimal transport
  plans} for \eqref{kant} will be denoted by
$\Gamma_{\mathrm{opt}}(\mu,\nu)$.
Notice that 
if there exists  $\ggamma=(\Id_{X_1}\times \tt)_{\#}\mu\in \Gamma(\mu,\nu)$, we have
\begin{equation*}
 \int_{X\times X}\sfd^2(x,y)\,\d\ggamma(x,y)=\int_X\sfd^2(x,\tt(x))\,\d\mu(x).
\end{equation*}
The quantity $W_2(\mu,\nu)$ defined by {\eqref{kant}} is a distance
between the measures $\mu,\nu\in\PP_2(X)$ which enjoys remarkable properties.
\begin{theorem}Let $(X,{\sfd})$ be a complete and separable metric
  space. Then, $W_2$ defines a distance on $\PP_2(X)$ and
  $(\PP_2(X),W_2)$ is a complete and separable metric space. Moreover, for a given sequence $\{\mu_k\}_{k\in\N}\subset\PP_2(X)$ we have
\begin{equation*}
\underset{k\freccia+\infty}{\lim}W_2(\mu_k,\mu)=0\quad\Leftrightarrow\quad\left\{\begin{aligned}
&\int_{X}f\,\d \mu_k\freccia\int_{X}f\,\d \mu\quad\forevery f\in C^0_b(X)\notag\\
&\underset{R\up+\infty}{\lim}\int_{X\barra B_R(x_0)}\sfd^2(x,x_0)\,\d \mu_k(x)=0\quad\text{uniformly w.r.t. $k\in\N$}.
\end{aligned}\right.
\end{equation*}  
\end{theorem}
The metric space $(\PP_2(X),W_2)$ is called the ($L^2$-)
\emph{Wasserstein space} on $X$.
When $X=\Rd$ we denote by $\PP_2^a(\R^d)$ the subset of $\PP_2(\R^d)$ defined by
\begin{equation}\label{pa2}
 \PP_2^a(\R^d):=\{\mu\in\PP_2(\R^d): \,\mu\ll\LL^d\}.
\end{equation}

Here we recall the following basic result on the existence and uniqueness of optimal transport plans induced by maps (which are then called \emph{optimal transport maps}) in the case in which the initial measure $\mu$ belongs to $\PP_2^a(\R^d)$.

\begin{theorem}[Existence and uniqueness of optimal transport maps,
  \cite{Knott-Smith84,Brenier91}]
    For any $\mu\in \PP_2^a(\R^d)$ and $\nu\in\PP_2(\R^d)$,
    Kantorovich's optimal transport problem \eqref{kant} has a unique
    solution $\ggamma$,
    which is concentrated on the graph of a transport map $\tt$.
    $\tt$ is the unique minimizer of Monge's optimal transport problem on $\R^d$ for the Euclidean distance
    \begin{equation*}
      \min\biggl\{\int_{\R^d}|x-\rr(x)|^2\,\d \mu(x):\,\rr_{\#}\mu=\nu\biggr\}.
    \end{equation*}
    The map $\tt$ is cyclically monotone and there exists a convex
    open set $\Omega\subset \R^d$ with $\mu(\R^d\setminus \Omega)=0$
    and a convex function $\phi:\Omega \to \R$ such that
    $\tt(x)=\nabla\phi(x)$ for $\mu$-a.e.\ $x\in \Omega$.
\end{theorem}

\subsubsection*{Geodesics and curvature properties of $(\PP_2(\R^d),W_2)$.}

\begin{theorem}[Geodesics in the Wasserstein space]
Given $\mu,\nu\in\PP_2(\R^d)$ and $\ggamma\in\Gamma_{\mathrm{opt}}(\mu,\nu)$, the curve 
\begin{equation*}
 [0,1]\ni s\mapsto \mu_s=\bigl((1-s)\pi^1+s\pi^2\bigr)_{\#}\ggamma.
\end{equation*}
is a constant speed geodesic between $\mu$ and $\nu$, i.e. it satisfies 
\begin{equation*}
 W_2(\mu_s,\mu_t)=|s-t|W_2(\mu_0,\mu_1)\quad\forevery s,t\in[0,1].
\end{equation*}
Vice versa, any constant speed geodesic between $\mu$ and $\nu$ can be built in this way.

\noindent If $\ggamma=(\Id\times\tt)_{\#}\mu$, then
\begin{equation*}
 \mu_s=\bigl((1-s)\Id+s\tt\bigr)_{\#}\mu,\quad s\in[0,1].
\end{equation*}
In particular, $(\PP_2(\R^d),W_2)$ is a geodesic space.
\end{theorem}


In view of the application to the Wasserstein framework of the theory of gradient flows in metric spaces developed in the previous section, we recall the following theorem (see Theorem 7.3.2 and Example 7.3.3 of \cite{Ambrosio-Gigli-Savare05})
\begin{theorem}[$(\PP_2(\R^d),W_2)$ is a PC-space]\label{teo:wasspc}
For any $\mu_0,\mu_1,\mu_2\in\PP_2(\R^d)$ we have
\begin{equation}\label{1conc} W_2^2(\mu_s,\mu_2)\geq(1-s)W_2^2(\mu_0,\mu_2)+sW_2^2(\mu_1,\mu_2)-s(1-s)W_2^2(\mu_0,\mu_1)\quad\forevery s\in[0,1],
\end{equation}
where $\mu_s$ is any constant speed geodesic between $\mu_0$ and $\mu_1$.

\noindent Moreover, when $d\geq2$ there is no constant $\lambda\in\R$ such that $W_2^2(\cdot,\mu_2)$ is $\lambda$-convex along geodesics.
\end{theorem}

According to Aleksandrov's notion of curvature for metric spaces, \eqref{1conc} can be interpreted by saying that the Wasserstein space {is} a \emph{positively curved }metric space (or PC-space). 

\noindent Then, the square of the Wasserstein distance along geodesics does not satisfy the $1$-convexity assumption \eqref{1conv}, which would be the most natural to prove the generation Theorem \ref{genteo} for the gradient flows of $\lambda$-convex functionals.

However, the theory developed in the previous section allows for a great flexibility in the choice of the connecting curves. In part{i}cular, for the Wasserstein space on $\R^d$ the $1$-convexity property \eqref{1conv} is satisfied along the following class of curves: 

\begin{definition}[Generalized geodesics] A generalized geodesic joining $\mu_2$ to $\mu_3$ (with base point $\mu_1$) is a curve of the type
\begin{equation}\label{gengeod}
 [0,1]\ni s\mapsto\mu_s^{2\freccia3}:=((1-s)\pi^2+s\pi^3)_{\#}\mmu,
\end{equation}
where
\begin{align}\label{mmu}
\mmu\in\Gamma(\mu_1,\mu_2,\mu_3)\quad\text{and}\quad {\pi^{1,2}}_{\#}\mmu\in\Gamma_{\mathrm{opt}}(\mu_1,\mu_2),\quad {\pi^{1,3}}_{\#}\mmu\in\Gamma_{\mathrm{opt}}(\mu_1,\mu_3).
\end{align}
Here $\Gamma(\mu_1,\mu_2,\mu_3):=\bigl\{\ggamma\in\PP(\R^d\times\R^d\times\R^d):\,\pi^i_{\#}\ggamma=\mu^i,\,i=1,2,3\bigr\}$ and $\pi^{i,j}:\R^d\times\R^d\times\R^d\freccia\R^d\times\R^d$ is the projection on the $i$-th and $j$-th coordinate. 
\end{definition}
\begin{proposition}[$1$-convexity of the Wasserstein distance along generalized geodesics] Let $\mu_1,\mu_2,\mu_3\in\PP_2(\R^d)$ and let $\mmu\in\Gamma(\mu_1,\mu_2,\mu_3)$ such that $\pi^{1,i}_{\#}\mmu\in\Gamma_{\mathrm{opt}}(\mu_1,\mu_i)$, for $i=2,3$. Then,
\begin{align*}
W_2^2\bigl(\mu_s^{2\freccia3},\mu_1\bigr)\leq(1-s)W_2^2(\mu_1,\mu_2)+sW_2^2(\mu_1,\mu_3)-s(1-s)W_2^2(\mu_2,\mu_3)\quad\forevery s\in[0,1].
\end{align*}
In particular, the function $\frac{1}{2}W_2^2(\mu_1,\cdot)$ is $1$-convex along generalized geodesics.
\end{proposition}

\end{subsection}

\begin{subsection}{Absolutely continuous curves in $(\PP_2(\R^d),W_2)$}

We recall here some basic properties of absolutely continuous curves
in the Wasserstein space, which are related to the ``dynamic
interpretation'' by \textsc{Benamou-Brenier} \cite{Benamou-Brenier00}.
The main result is the following  \cite[Theorem 8.3.1]{Ambrosio-Gigli-Savare05}:
\begin{theorem}[Absolutely continuous curves and the continuity
  equation]\label{teo:ac}
Let $\mu_t:(0,+\infty)\freccia\PP_2(\R^d)$ be an absolutely continuous curve and let $|\mu\sp{\prime}|\in\rL^1(0,+\infty)$ be its metric derivative. Then there exists a Borel vector field ${\vv}:(x,t)\mapsto {\vv}_t(x)$ such that
\begin{equation}\label{normvt}
 {\vv}_t\in\rL^2(\mu_t;\R^d),\quad||{\vv}_t||_{\rL^2(\mu_t;\R^d)}\leq|\mu\sp{\prime}|(t)\quad\text{for $\LL^1$-a.e. $t\in(0,+\infty)$}
\end{equation}
and the continuity equation
\begin{equation}\label{conteq}
 \frac{\partial}{\partial t}\mu_t+\nabla\cdot(\vv_t\mu_t)=0\quad\text{in $\R^d\times(0,+\infty)$}
\end{equation}
holds in the sense of distributions. 

\noindent Moreover, 
\begin{equation}\label{veltan}
\vv_t\in\overline{\{\nabla\varphi:\,\varphi\in C^{\infty}_c(\R^d)\}}^{\rL^2(\mu_t;\R^d)}\quad\text{for $\LL^1$-a.e. $t\in(0,+\infty)$.}
\end{equation}
\noindent Conversely, if a curve $(0,+\infty)\ni t\mapsto \mu_t\in\PP_2(\R^d)$ is continuous w.r.t. the weak topology on $\PP(\R^d)$ and it satisfies the continuity equation \eqref{conteq} for some Borel vector field $\vv_t$ with
\begin{equation*}
 \int_0^{+\infty}||\vv_t||_{\rL^2(\mu_t;\R^d)}<+\infty,
\end{equation*}
then $\mu_t$ is an absolutely continuous curve and $|\mu\sp{\prime}|(t)\leq||\vv_t||_{\rL^2(\mu_t;\R^d)}$ for $\LL^1$-a.e. $t\in(0,+\infty)$.

\end{theorem}

Then, the minimal norm for the vector fields $\vv_t$ satisfying \eqref{conteq} for an absolutely continuous curve $\mu_t$ is given by its metric derivative. Furthermore, such ``minimal'' vector fields satisfy \eqref{veltan}. 

\noindent This fact suggests the following definition (we refer to
\cite[Chap.~8.4]{Ambrosio-Gigli-Savare05}).
\begin{definition}[Tangent space]Let $\mu\in\PP_2(\R^d)$. We define the \emph{tangent space} to $\PP_2(\R^d)$ at the point $\mu$ as
\begin{equation}\label{tanspace}
 \mathrm{Tan}_{\mu}\PP_2(\R^d):=\overline{\{\nabla\varphi:\,\varphi\in C^{\infty}_c(\R^d)\}}^{\rL^2(\mu_t;\R^d)}.
\end{equation}
 \end{definition}

\begin{proposition}[Tangent vectors to absolutely continuous curves]\label{tangvec} Let $\mu_t:(0,+\infty)\freccia\PP_2(\R^d)$ be an absolutely continuous curve and let $\vv_t\in\rL^2(\mu_t;\R^d)$ be a Borel vector field such that \eqref{conteq} holds. 
\noindent Then $\vv_t$ satisfies \eqref{normvt} if and only if $\vv_t\in \mathrm{Tan}_{\mu_t}\PP_2(\R^d)$ for $\LL^1$-a.e. $t\in(0,+\infty)$.
\noindent The vector $\vv_t$ is uniquely determined $\LL^1$-a.e. by \eqref{normvt} and \eqref{conteq}.
\end{proposition}
Tangent vector fields are also strictly related to the first order infinitesimal behavior of the Wasserstein distance along absolutely continuous curves.

\begin{proposition} Let $\mu_t:(0,+\infty)\freccia\PP_2(\R^d)$ be an absolutely continuous curve and let $\vv_t\in\mathrm{Tan}_{\mu_t}\PP_2(\R^d)$ be the tangent vector characterized by Proposition \ref{tangvec}. Then, for $\LL^1$-a.e. $t\in(0,+\infty)$ the following property holds:
\begin{align}\label{wassinf}
\underset{h\freccia0}{\lim }\,\frac{W_2(\mu_{t+h},(\Id+h\vv_t)_{\#}\mu_t)}{|h|}=0.
\end{align}
Then, if $\mu_t$ and $\mu_{t+h}$ are linked by an optimal transport map ${\tt}_{\mu_t}^{\mu_{t+h}}$, we have 
\begin{equation*}
 \underset{h\freccia0}{\lim}\,\frac{\tt_{\mu_t}^{\mu_{t+h}}-\Id}{h}=\vv_t\quad\text{in $\rL^2(\mu_t;\R^d)$}.
\end{equation*}
\end{proposition}
As an application of \eqref{wassinf} we are able to show the $\LL^1$-a.e. differentiability of $t\mapsto W_2(\mu_t,\sigma)$ along absolutely continuous curves $\mu_t$ 
in terms of tangent vectors and optimal transport plans; this provides
a useful formula for the left hand side of the \eqref{eq:EVI}
\begin{equation*}
  \frac{1}{2}\frac{\d}{\dt}W_2^2(\mu_t,\sigma)\leq\phi(\sigma)-\phi(\mu_t)\quad\forevery \sigma\in D(\phi)
\end{equation*}

\begin{theorem}Let $\mu_t:(0,+\infty)\freccia\PP_2(\R^d)$ be an absolutely continuous curve, let $\vv_t\in\mathrm{Tan}_{\mu_t}\PP_2(\R^d)$ be the tangent vector characterized by Proposition \ref{tangvec} and let $\sigma\in\PP_2(X)$. Then
\begin{equation}\label{wassder}
  \frac{1}{2}\frac{\d}{\dt}W_2^2(\mu_t,\sigma)=\int_{\R^d\times \R^d}\langle x-y,\vv_t(x)\rangle\,\d\ggamma(x,y)\quad\forevery \ggamma\in\Gamma_{\mathrm{opt}}(\mu_t,\sigma),
\end{equation}
for $\LL^1$-a.e $t\in(0,+\infty)$.
\noindent In particular, if $\mu_t\in\PP_2^a(\R^d)$,
\begin{equation}\label{wassderac}
  \frac{1}{2} \frac{\d}{\dt}W_2^2(\mu_t,\sigma)=\int_{\R^d}\langle x-\tt_{\mu_t}^{\sigma}(x),\vv_t(x)\rangle\,\d \mu_t(x),
 \end{equation}
where $\tt_{\mu_t}^{\sigma}$ is the unique optimal transport map between $\mu_t$ and $\sigma$.
\end{theorem}

\end{subsection}

\begin{subsection}{Geodesically $\lambda$-convex functionals in $\PP_2(\R^d)$}
In this section we introduce the three main classes of $\lambda$-geodesically convex functionals on the Wasserstein space $(\PP_2(\R^d),W_2)$
introduced by \textsc{McCann} \cite{McCann97}
(for the proofs of the main results we refer to Chapter 9 of \cite{Ambrosio-Gigli-Savare05}).

\begin{example}[Potential energy]
  \label{ex:potential}
  Let $V:\R^d\freccia\R$ be a $\lambda_V$-convex function for some $\lambda_V\in \R$
  and let us define the potential energy
\begin{equation}\label{potential}
 \VV(\mu):=\int_{\R^d}V(x)\,\d \mu(x)\quad \text{for every }\mu\in \PP_2(\R^d).
\end{equation}
For every $\mu_1,\mu_2\in D(\VV)$ and $\mmu\in\Gamma(\mu_1,\mu_2)$ we have
\begin{equation}
\VV\bigl(\bigl[(1-s)\pi^1+s\pi^2\bigr]_{\#}\mmu\bigr)\leq(1-s)\VV(\mu_1)+s\VV(\mu_2)-\frac{\lambda_V}{2}s(1-s)\int_{\R^d\times\R^d}|x-y|^2\,\d \mmu(x,y).
\end{equation}
In particular, $\VV$ is geodesically $\lambda_V$-convex on $\PP_2(\R^d)$.
\end{example}
\begin{example}[Interaction energy]
  \label{ex:interaction}
  Let $\lambda_W\le 0$ and
  let $W:\R^d\to\R$ be a $\lambda_W$-convex function with $W(-x)=W(x)$ for every $x\in \R^d$,
  and let us set
  \begin{equation}
    \label{eq:41}
    \WW(\mu):=\frac 12\iint_{\R^d\times\R^d} W(x-y)\,\d\mu(x)\,\d\mu(y)\quad \text{for every }\mu\in \PP_2(\R^d).
  \end{equation}
  For every $\mu_1,\mu_2\in D(\WW)$ and $\mmu\in\Gamma(\mu_1,\mu_2)$ we have
  \begin{equation}  \WW\bigl(\bigl[(1-s)\pi^1+s\pi^2\bigr]_{\#}\mmu\bigr)\leq
    (1-s)\WW(\mu_1)+s\WW(\mu_2)-\frac{\lambda_W}{2}s(1-s)\int_{\R^d\times\R^d}|x-y|^2\,\d \mmu(x,y).
\end{equation}
In particular, $\WW$ is geodesically $\lambda_W$-convex on $\PP_2(\R^d)$.
\end{example}
\begin{example}[Internal energy]
  \label{ex:internal}
  Let $U:[0,+\infty)\freccia \R$ be a convex function such that
  \begin{align}\label{uhp}
    &U(0)=0,\qquad
    \liminf_{s\downarrow0}\frac{U(s)}{s^{\alpha}}>-\infty
    \quad\text{for some $\alpha>\frac{d}{d+2}$},\qquad
    \underset{s\freccia+\infty}{\lim}\,\frac{U(s)}{s}=+\infty\\
    &\label{uconv}
    \text{the map $s\mapsto s^dU(s^{-d})$ is convex and non-increasing on $(0,+\infty)$}.
\end{align}
The internal energy functional
\begin{equation}\label{uint}
 \UU(\mu):=
\begin{cases}
&\displaystyle\int_{\R^d}U(\rho(x))\,\d \LL^d(x)\quad\text{if $\mu\ll\LL^d$, $\rho=\frac{d\mu}{d\LL^d}$},\\
&+\infty\quad\text{otherwise}.
\end{cases}
\end{equation}
is geodesically convex and lower semicontinuous in $\PP_2(\R^d)$.
Among the functionals $\UU$ with $U$ satisfying \eqref{uhp} and \eqref{uconv} we have
\begin{align}
&\text{the entropy functional }U(s)=s\log s\label{entropy}\\
&\text{the power functional }U(s)=\frac{s^m}{m-1},\quad m>1.\label{power}
\end{align}
Property \eqref{uconv} is also satisfied in the range $1-\frac 1d<m<1$; however, since in this case
$U$ is not superlinear at infinity (the third condition of \eqref{uhp}),
we have to consider the relaxed lower semicontinuous functional
\begin{equation*}
 \UU^{*}(\mu)=\int_{\R^d}\frac{\rho^m}{m-1}\,\d \LL^d, \quad
 \text{if }\mu=\rho\LL^d+\mu^\perp,\ \mu^\perp\perp\LL^d.
\end{equation*}
\end{example}
\begin{example}[Relative entropy]
  \label{ex:relative}
  Let $\mu,\gamma$ be two measures in $\PP(\R^d)$. Then, the relative entropy of $\mu$ w.r.t. $\gamma$ is the functional defined by
\begin{equation}\label{ent}
 \mathrm{Ent}_{\gamma}(\mu):=
\begin{cases}
  &\displaystyle \int\frac{d\mu}{d\gamma}\log\frac{d\mu}{d\gamma}\,\d \gamma,\quad\text{if $\mu\ll\gamma$},\\
&+\infty\quad\text{otherwise.}
\end{cases}
\end{equation}
\end{example}
We note that \eqref{ent} corresponds to the internal energy associated to the function \eqref{entropy}
when $\gamma=\LL^d$ (which nevertheless is not in $\PP(\R^d)$).
\begin{proposition}\label{prop:ent} The relative entropy $ \mathrm{Ent}_{\gamma}$ is geodesically convex in $\PP_2(\R^d)$ if and only if one of the following conditions holds:
\begin{align}
\text{\rm 1. }&\gamma=\rme^{-V}\LL^d\text{ for some convex function $V:\R^d\freccia\R$;}\\
\text{\rm 2. }&\gamma\text{ is \emph{log-concave}, i.e. for every couple of open sets $A,B\subset\R^d$, $t\in[0,1]$}\notag\\
\label{eq:54}
&\log\gamma((1-t)A+tB)\geq(1-t)\log\gamma(A)+t\log\gamma(B).
\end{align}
\end{proposition} 
Now we introduce the notion of convexity which will be crucial to apply the metric theory of gradient flows developed in Section \ref{sect:genteo} to the main examples
\ref{ex:potential}, \ref{ex:interaction}, \ref{ex:internal}, \ref{ex:relative} of
geodesically $\lambda$-convex functionals in the Wasserstein space.
\begin{definition}[Convexity along generalized geodesics] Given $\lambda\in\R$, we say that $\phi:\PP_2(\R^d)\freccia(-\infty,+\infty]$ is $\lambda$-convex along generalized geodesics if for any $\mu_1,\mu_2,\mu_3\in D(\phi)$ there exists a generalized geodesic $[0,1]\ni s\mapsto\mu_s^{2\freccia3}$ joining $\mu_2$ to $\mu_3$ induced by a plan $\mmu$ satisfying \eqref{mmu} such that
\begin{align}\label{genconv}
\phi(\mu_s^{2\freccia3})\leq(1-s)\phi(\mu_2)+s\phi(\mu_3)-\frac{\lambda}{2}s(1-s)\int
|x_2-x_3|^2\,\d \mmu(x_1,x_2,x_3)\quad\forevery s\in[0,1].
\end{align}
\end{definition}
\begin{lemma}\label{convphiw2}$[(1\slash\ttau+\lambda)$-convexity of $\Phi(\ttau,\mu_1;\cdot)]$ Let $\phi:\PP_2(\R^d)\freccia(-\infty,+\infty]$ be a proper functional which is $\lambda$-convex along generalized geodesics for some $\lambda\in\R$. Then, for each $\mu_1\in D(\phi)$ and for each $0<\ttau<\frac{1}{\lambda^-}$ the functional
\begin{equation*}
 \Phi(\ttau,\mu_1;\mu):=\frac{1}{2\ttau}W_2^2(\mu_1,\mu)+\phi(\mu) \text{ satisfies the convexity assumption \eqref{convpphi}.}
\end{equation*}
 \end{lemma}
By Lemma \ref{convphiw2}, whenever $\phi$ is proper, l.s.c.~and $\lambda$-convex along generalized geodesics in $\PP_2(\R^d)$ we can apply Theorem \ref{genteo} and get the existence, uniqueness and regularizing estimates for the solutions of the {$\EVIshort{\lambda}$}.

The examples of geodesically convex functionals in $\PP_2(\R^d)$ which have been introduced in this section are also convex along the generalized geodesics.
\begin{theorem}\label{fconvgen}
  The functionals on $\PP_2(\R^d)$ considered by the examples \ref{ex:potential},
  \ref{ex:interaction} (with $\lambda\le0$),
  \ref{ex:internal} (under condition \eqref{uconv}), and \ref{ex:relative}
  (under condition \eqref{eq:54}) are $\lambda$-convex along generalized geodesics.
\end{theorem}  

\end{subsection}

\begin{subsection}{Gradient flows in $(\PP_2(\R^d),W_2)$ and evolutionary PDE's}

In this section we show some applications of the generation result \ref{genteo}
to the existence, well-posedness, and asymptotic behavior of nonnegative solutions $\rho:\R^d\times(0,+\infty)\freccia\R$ of evolutionary PDE's of the type
\begin{equation}\label{df}
 \frac{\partial}{\partial t}\rho-\nabla\cdot\biggl(\rho\nabla\frac{\delta \phi}{\delta \rho}\biggr)=0\quad\text{in }\R^d\times(0,+\infty),
\end{equation}
where
$ \frac{\delta \phi(\rho)}{\delta \rho}$ 
is the first variation of a suitable integral functional; here we
consider the case of functionals which are a positive linear combination of the three kinds of
contributions considered in the {E}xamples \ref{ex:potential}, \ref{ex:interaction}, and \ref{ex:internal},
i.e. $  \phi(\rho):=
\alpha_1\UU(\rho)+\alpha_2\VV(\rho)+\alpha_3\WW(\rho) $ where $\alpha_i\ge0$ and
\begin{equation}
  \label{eq:36}
  \begin{aligned}
    \UU(\rho):=&\int_\Rd U(\rho(x))\,\d x,\\
    \VV(\rho):=&\int_\Rd
    V(x)\rho(x)\, \d x,\\
    \WW(\rho):=&\frac 12\int_{\Rd\times \Rd}
    W(x-y)\rho(x)\rho(y)\,\d x\d y,
  \end{aligned}
\end{equation}
so that
\begin{equation}
  \label{eq:37}
  \frac{\delta \phi(\rho)}{\delta \rho}=\alpha_1U'(\rho)+\alpha_2V+\alpha_3W*\rho.
\end{equation}
In the particular cases of the Fokker-Planck equation ($\phi=\UU+\VV$ and
$U(r)=r\log r$)
\begin{equation}\label{FP}
 \frac{\partial}{\partial t}\rho-\nabla\cdot(\nabla\rho+\rho\nabla V)=0\quad\text{in }\R^d\times(0,+\infty),
\end{equation}
and of the nonlinear diffusion equations ($\phi=\UU$, $U(r)=\frac1{m-1}r^m$)
\begin{equation}\label{porous}
 \frac{\partial}{\partial t}\rho-\Delta\rho^m=0,\quad m\geq1-\frac{1}{d},
\end{equation}
the Wasserstein approach has been introduced by the remarkable papers of \textsc{Jordan-Kinderlehrer-Otto}
\cite{Jordan-Kinderlehrer-Otto98} and \textsc{Otto} \cite{Otto01} and then extended in many interesting directions,
covering a wide range of application{s}: see e.g.
\cite{Otto-Villani01,Agueh02,Carlen-Gangbo03,Carlen-Gangbo04,Carrillo-McCann-Villani01,Evans-Gangbo-Savin05,Otto-Westdickenberg05,Carrillo-McCann-Villani06,Carrillo-DiFrancesco-Lattanzio06,Ambrosio-Serfaty08,Blanchet-Calvez-Carrillo08,Fang-Shao-Sturm08,Ambrosio-Savare-Zambotti09,Carrillo-Lisini-Savare-Slepcev09,Natile-Savare09,Matthes-McCann-Savare09,Natile-Peletier-Savare-preprint10}.
  
The results presented here are just examples of the transport approach{.}

\begin{theorem}\label{teo:gfwass}
  Let $V,W,U$ be as in the examples \ref{ex:potential}, \ref{ex:interaction} and \ref{ex:internal},
  let $\VV,\WW,\UU$ be defined as in \eqref{eq:36}, and let
  $\phi:=\alpha_1\UU+\alpha_2\VV+\alpha_3\WW$.
  For every $\mu_0\in \PP_2(\R^d)$ there exists a unique solution $\mu_t\in \mathrm{Lip}_{\rm loc}(0,+\infty;\PP_2(\R^d))$
  satisfying $\EVI{\PP_2(\R^d)}{W_2}\phi\lambda$, $\lambda:=\alpha_2\lambda_V+\alpha_3\lambda_W$,
\begin{equation}
  \frac{1}{2}\frac{\d}{\dt}W_2^2(\mu_{t},\sigma)\leq\phi(\sigma)-\phi(\mu_{t})
  -\frac\lambda2W_2^2(\mu_t,\sigma)\quad\forevery\sigma\in D(\phi)
  \label{eq:45}
\end{equation}
with $\underset{t\downarrow0}{\lim}\,\mu_{t}=\mu_0$ in $\PP_2(\R^d)$;
the curve $\mu$ satisfies all the properties stated in Theorem \ref{thm:main1}, the continuity equation
\begin{gather}
  \label{eq:46}
  \frac\partial{\partial t}\mu_t+\nabla\cdot(\mu_t\,\vv_t)=0\quad \text{in }\R^d\times (0,+\infty),\quad\text{with}\\
  \vv_t\in \mathrm{Tan}_{\mu_t}\PP_2(\R^d)\quad\text{$\LL^1$-a.e.\ in $(0,+\infty)$ and}\quad
  t\mapsto \int_{\R^d}|\vv_t|^2\,\d\mu_t=|\mu_t'|^2\in L^\infty_{\rm loc}(0,+\infty),
\end{gather}
and for $\LL^1$-a.e.\ $t\in (0,+\infty)$
the velocity vector $\vv_t\in \mathrm{Tan}_{\mu_t}\PP_2(\R^d)$ satisfies the ``subdifferential inequality''
\begin{equation}
  \label{eq:47}
  \int_{\R^d}\la \vv_t(x),x-y\ra+\frac \lambda2|y-x|^2\,\d\ggamma_t(x,y)\le \phi(\sigma)-\phi(\mu_t)
  \quad\text{for every }\ggamma_{{t}}\in \Gamma_{\rm opt}(\mu_t,\sigma).
\end{equation}
\end{theorem}
We can give an explicit characterization of the system \eqref{eq:46}, \eqref{eq:47}.
Here we consider the simpler case when $U,V,W$ are differentiable and satisfy a
\emph{doubling condition}: for a function $f:\R^h\to \R$ it means that there exists a constant
$C>0$ such that
\begin{equation}
  \label{eq:42}
  f(x+y)\le C(1+f(x)+f(y))\quad\text{for every }x,y\in \R^h.
\end{equation}
We also set
\begin{equation}
  \label{eq:43}
  L_U(r):=rU'(r)-U(r)\quad \text{if }r>0,\quad L_U(0)=0.
\end{equation}
\begin{theorem}
  Under the same assumption{s} of the previous theorem, let us also suppose that
  $U,V,W$ are differentiable and satisfy the doubling condition \eqref{eq:42}.
  The locally Lipschitz curve $\mu$ characterized by \eqref{eq:45} (or by \eqref{eq:46}, \eqref{eq:47})
  solves the following evolutionary PDE's in $\R^d\times(0,+\infty)$\\
  \textbf{Transport equation,} $\phi=\VV$, $\vv_t=-\nabla V$:
  \begin{equation}
    \label{eq:48}
    \frac\partial{\partial_t}\mu_t-\nabla\cdot(\mu_t\nabla V)=0.
  \end{equation}
  \textbf{Nonlocal interaction equation,} $\phi=\WW$, $\vv_t=-(\nabla W)\ast \mu_t$
  \begin{equation}
    \label{eq:49}
    \frac\partial{\partial_t}\mu_t-\nabla\cdot(\mu_t(\nabla W\ast\mu_t))=0.
  \end{equation}
  \textbf{Fokker-Planck equation,} $\phi=\UU+\VV$, $U(r)=r\log r$, $-\mu_t\vv_t=\nabla \mu_t+\mu_t \nabla V$
  \begin{equation}
    \label{eq:50}
      \frac\partial{\partial_t}\mu_t-\nabla\cdot(\nabla\mu_t+\mu_t\nabla V)=0.
  \end{equation}
  In this case, $\mu_t=\rho_t\LL^d$ with $\rho_t\in W^{1,1}_{\rm loc}(\R^d)$ for $\LL^1$-a.e.\ $t\in (0,+\infty)$.\\
  \textbf{Nonlinear diffusion equation,} $\phi=\UU$,
  $\mu_t\vv_t=-\nabla L_U(\rho_t)$ where $\mu_t=\rho_t\LL^d\ll\LL^d$,
  \begin{equation}
    \label{eq:44}
    \frac\partial{\partial_t}\mu_t-\Delta(L_U(\rho_t))=0,\quad 
  \end{equation}
  with $L_U(\rho_t)\in W^{1,1}_{\rm loc}(\R^d)$  for $\LL^1$-a.e.\
  $t\in (0,+\infty)$.\\
  \textbf{Drift-diffusion with non local interactions,}
  $\phi=\UU+\VV+\WW$, $-\mu_t\vv_t=\nabla L_U(\rho_t)+{\mu_t}\nabla V+{\mu_t((\nabla
  W)\ast\mu_t)}$, $\mu_t=\rho_t\LL^d\ll\LL^d$,
  \begin{equation}
    \label{eq:51}
    \frac\partial{\partial_t}\mu_t-\nabla\cdot\big(\nabla L_U(\rho_t)+\mu_t\nabla
    V+\mu_t((\nabla W)\ast\mu_t)\big)=0
  \end{equation}
\end{theorem} 
We refer to \cite[Chap.~11]{Ambrosio-Gigli-Savare05} for the proofs
and for more general and detailed results;
here we just give a sketch of the argument showing that \eqref{eq:46},
\eqref{eq:47} yield \eqref{eq:50} when $\phi=\UU+\VV$ in the case of
$U(r)=r\log r$.


Let us fix a time $t>0$ where \eqref{eq:47} holds, a smooth test function $\zeta\in C^\infty_{\rm c}(\R^d)$,
and $\tt_\eps:=\ii+\eps\nabla \zeta$. If
$|\eps|\max_{\R^d}\|\rmD^2\zeta\|<1$ the coupling
$\ggamma_\eps:=(\ii,\tt_\eps)_\#\mu_t$ is optimal between $\mu_t$ and
$(\tt_\eps)_\#\mu_t$ so that \eqref{eq:47} yields
\begin{equation*}
  -\eps\int_{\R^d}\langle\vv_t(x), \nabla\zeta(x)\rangle\,\d
  \mu_t(x)\leq \phi((\tt_\eps)_\#\mu_t)-\phi(\mu_t).
\end{equation*}
Setting
\begin{equation*}
  \rho_t:=\frac{\d\mu_t}{\d\LL^d},\quad\rho_t^{\varepsilon}:=\frac{\d(\tt_\eps)_{\#}\mu_t}{\d\LL^d}
\end{equation*}
we get
\begin{align*}
  -\eps\int_{\R^d}\langle\vv_t, \nabla\zeta\rangle\,\d
  \mu_t
  \leq \int_{\R^d}\rho_t^{\varepsilon}\log\rho_t^{\varepsilon}\,\d
  \LL^d-\int_{\R^d}\rho_t\log\rho_t\,\d \LL^d
  +\int_{\R^d} \big(V(\tt_\eps(x))-V(x)\big)\,\d\mu_t(x)
 \end{align*}
Applying the change of variables formula
\begin{equation*}
 \rho_t^{\varepsilon}(\tt_\eps(x))\,\mathrm{det}[\Id+\varepsilon \rmD^2\zeta(x)]=\rho_t(x),
\end{equation*}
we obtain
\begin{equation}\label{heatineq}
  -\eps\int_{\R^d}\langle\vv_t,\nabla\zeta\rangle\,\d\mu_t\leq
  -\int_{\R^d}\rho(x)\log\bigl(\mathrm{det}[\Id+\varepsilon
  \rmD^2\zeta(x)]\bigr)\,\d \LL^d
  +\int_{\R^d} \big(V(\tt_\eps(x))-V(x)\big)\,\d\mu_t(x)
\end{equation}
Finally, dividing by $\eps$ and taking the limit of \eqref{heatineq} as $\varepsilon$ tends to $0$ we get
\begin{equation*}
    -\int_{\R^d}\langle\vv_t,\nabla\zeta\rangle\,\d\mu_t=
    \int_{\R^d} \Big(-\Delta \zeta(x)+\nabla V(x)\cdot \nabla\zeta\Big)\,\d\mu_t\quad\forevery \zeta\in C^{\infty}_c(\R^d),
\end{equation*}
so that $\mu$ satisfies the distributional formulation of \eqref{eq:50}.
\end{subsection}

\subsection{The heat flow on Riemannian manifolds and metric-measure spaces}
We conclude these notes by giving a short account of possible applications of the
Wasserstein setting to the generation of the heat flow in Riemannian
manifolds and metric-measure spaces.

Let us start with a compact and smooth Riemannian manifold
$(M,g)$; we denote by $\sfd_g$ its Riemannian distance and by $\gamma=\mathrm{Vol}_g\in
\PP(M)$ its
(normalized) volume measure.

In $\PP_2(M)$ we consider the Relative Entropy functional
$\mathrm{Ent}_\gamma$ as in \eqref{ent}. \textsc{Von Renesse-Sturm}
\cite{Sturm-VonRenesse05}
proved
\begin{theorem}
  \label{thm:VRS}
  The Relative Entropy functional $\mathrm{Ent}_\gamma$ is
  geodesically $\lambda$-convex in $\PP_2(M)$ if and only if $M$
  satisfies the lower Ricci curvature bound
  \begin{equation}\label{eq:52}
    \mathrm{Ric}(M)\ge
    \lambda\quad
    \text{i.e.\quad $\mathrm{Ric}_x(v,v)\ge \lambda |v|^2_g$\quad for all
      $x\in M$ and $v\in \mathrm{Tan}_x(M)$.}
  \end{equation}
\end{theorem}
{In this case, it is} possible to show (see
\cite{Otto-Villani00,Otto-Westdickenberg05,Daneri-Savare08,Ohta09,Villani09,Erbar10})
that the Relative Entropy
functional $\mathrm{Ent}_\gamma$ generates a $\lambda$-gradient flow
$\sfS_t:\PP_2(M)\to \PP_2(M)$ 
according to definition \ref{def:GFlow}, which coincides with the
classical heat flow on $M$.
\begin{theorem}
  The relative entropy functional $\mathrm{Ent}_\gamma$ generates
  a $\lambda$-gradient flow $\sfS_t$ in $\PP_2(M)$ according to
  Definition \ref{def:GFlow} (and thus satisfying all the properties
  stated in Theorems
  \ref{thm:main1} and \ref{thm:Error_Estimate}).
  A curve $\mu_t\in \PP_2(M)$ is a
  solution 
  of $\EVI M{{\sfd_g}}{\mathrm{Ent}_\gamma}\lambda$ if and only if its density
  $\rho_t=\d\mu_t/\d\gamma$ solves the Heat equation
  \begin{equation*}
    \frac{\partial}{\partial t}\rho_t-\Delta_g\,\rho_t=0\quad\text{in $M\times(0,+\infty)$},
  \end{equation*}
  where $\Delta_g$ is the Laplace-Beltrami operator on $M$.
\end{theorem}
The adimensionality of the form of the Entropy functional \eqref{ent} and the purely metric character of the EVI suggest that one can use them to define a heat flow on more general \emph{measure-metric} spaces $(X,\sfd,\gamma)${,} where
$(X,\sfd)$ is a complete and separable metric space and $\gamma\in\PP(X)$. Indeed, as it has been often pointed out in the previous sections, the EVI formulation gives nice regularity, stability and asymptotic properties for the related flow.
We briefly sketch two possible approaches:

\subsubsection*{Approximation by measured Gromov-Hausdorff convergence.}
We consider a sequence of smooth and compact Riemannian manifold{s} $(M^h,\sfd^h,\mathrm{Vol}^h)$ converging to a limit
measure-metric space $(X,\sfd,\gamma)$ in the measured Gromov-Hausdorff convergence: it means
\cite{Sturm05} that
a sequence $\{{\hat\sfd}^h\}_{k\in \N}$ of (complete, separable) \emph{coupling semidistances}  on the disjoint union
$M^h\sqcup X$ exists such that the restriction of $\hat\sfd^h$ on $M^h$ (resp.~$X$) coincides with
$\sfd^h$ (resp.~$\sfd$) and
\begin{equation}
  \label{eq:53}
  \lim_{k\up\infty} \hat W_{2}^h(\mathrm{Vol}^h,\gamma)=0,\quad
  \hat W_{2}^h\text{ is the Wasserstein distance on }\PP_2(M^h\sqcup X) \text{ induced by }{\hat\sfd}^h
\end{equation}
A sequence $\mu^h\in \PP_2(M^h)$ converges to $\mu\in \PP_2(X)$ if $\lim_{k\up+\infty}\hat W_{2}^h(\mu^h,\mu)=0$.
Adapting the arguments of Theorem \ref{thm:main_stability1}
it is possible to prove the following asym{p}totic result:
\begin{theorem}[\cite{Savare07}]
  Let us assume that the compact Riemannian manifolds $M^h$ satisfy the uniform lower bound on the Ricci curvature
  $\mathrm{Ric}(M^h)\ge\lambda$ for some $\lambda\in \R$ independent of $k$ and converge
  to $(X,\sfd,\gamma)$ in the measured Gromov-Hausdorff sense.
  Then the Relative Entropy functional $\mathrm{Ent}_\gamma$ admits a $\lambda$-gradient flow $\sfS_t$
  on $\PP_2(X)$ and for every sequence of initial measures $\mu_{0}^h\in \PP_2(M^h)$ converging to
  $\mu_0\in \PP_2(X)$ the corresponding solution $\mu_t^h$ of the Heat flow on $M^h$
  converges to $\sfS_t(\mu_0)$ in $\PP_2(X)$ for every $t>0$.
\end{theorem}
Applying Theorem \ref{eviconv} one finds in particular that the limit Entropy functional
$\mathrm{Ent}_\gamma$ is \emph{strongly}
geodesically $\lambda$-convex (at least when the support of $\gamma$ is $X$),
a stability result that has been proved by \cite{Sturm06I,Lott-Villani09}.

\subsubsection*{Intrinsic costruction}
Starting from Theorem \ref{thm:VRS}, \textsc{Sturm} \cite{Sturm06I} and \textsc{Lott-Villani}
\cite{Lott-Villani09} introduced the concept of metric-measure spaces $(X,\sfd,\gamma)$ satisfying
a lower Ricci curvature bound, by requiring that the relative entropy functional
$\mathrm{Ent}_\gamma$ is geodesically $\lambda$-convex in $\PP_2(X)$.
\begin{definition}[Lower Ricci curvature bounds for metric-measure spaces]
  \label{def:Ric}
  We say that a metric-measure space $(X,{\sfd},\gamma)$ has Ricci curvature bounded from below by a certain $\lambda\in\R$ (and we write $\mathrm{Ric}(X)\geq\lambda$) if the relative entropy $\mathrm{Ent}_{\gamma}$ is $\lambda$-geodesically convex on $X$.
\end{definition}

It is then natural to look for other intrinsic properties of $X$ which are sufficient to deduce
the existence of the associated EVI semigroup. It is interesting to notice that if the
relative entropy functional generates a $\lambda$-gradient flow $\sfS_t$ then $\sfS_t$ is {a}
semigroup of \emph{linear} operators \cite{Savare07}.
In the case of compact positively curved (PC) Alexandrov spaces
the existence of a $\lambda$-contracting gradient flow can be deduced by a general
unpublished result of \cite{Perelman-Petrunin} and has been recently proved by
\textsc{Ohta} \cite{Ohta07}.

In more general cases, we can apply Theorem \ref{genteofinal}:
\begin{theorem}
  Let us suppose that $(X,\sfd,\gamma)$ is a complete and separable metric-measure space with
  Ricci curvature bounded from below, according to Definition \ref{def:Ric}, and measure $\gamma$ with full support
  $\supp(\gamma)=X$.
  If $X$ satisfies the Local Angle Condition \ref{defi:lac} and it is $\mathrm{K}$-semiconcave {as in} \ref{defi:ksc}, then
  the relative entropy functional $\mathrm{Ent}_\gamma$ generates
  a $\lambda$-gradient flow on $\PP_2(X)$ which can be uniquely extended to a Markov semigroup
  (i.e.\ linear, order preserving, strongly continuous, contractive)
  in every space $L^p(\gamma)$, $p\in [1,+\infty)$.  
\end{theorem}


\def\cprime{$'$} \def\cprime{$'$}

\end{document}